\documentclass[12 pt]{article}
\pdfoutput=1

\usepackage{geometry}
\usepackage[utf8]{inputenc}
\usepackage[T1]{fontenc}

\usepackage{amsmath}
\usepackage{amssymb}
\usepackage{amsthm}
\usepackage{mathtools}
\usepackage{mathrsfs}
\usepackage{dsfont}
\usepackage{stmaryrd}

\usepackage{csquotes} 
\usepackage{xcolor} 
\usepackage{soul} 
\usepackage{framed} 

\usepackage{enumitem} 

\usepackage{graphicx} 
\usepackage{float}
\usepackage{caption}
\usepackage{tikz}
\usetikzlibrary{shapes.geometric}
\usetikzlibrary{arrows}

\usepackage{tikz-cd}

\usepackage{pgfplots}
\pgfplotsset{compat=1.18}

\usepackage[backend=biber,
            style=alphabetic,
            maxnames=10,
            minnames=2,
            maxalphanames=5,
            minalphanames=2,
            backref=true,
            backrefstyle=none]{biblatex}

\usepackage[colorinlistoftodos]{todonotes} 
\usepackage{pdflscape} 

\usepackage{hyperref}
\hypersetup{
    breaklinks=true,
    colorlinks=true,
    linkcolor=cyan,     
    urlcolor={blue!60},
    citecolor=teal,
}

\usepackage[hyphenbreaks]{breakurl} 

\usepackage{adjustbox} 

\usepackage{totcount} 

\usepackage{upgreek} 

\numberwithin{equation}{section}

\theoremstyle{plain} 
\newtheorem{claim}[equation]{Claim}
\newtheorem{prop}[equation]{Proposition}
\newtheorem{cor}[equation]{Corollary}

\newtheorem{lemma}[equation]{Lemma}
\newtheorem{rmk}[equation]{Remark}

\newtheorem{hypothesis}[equation]{Hypothesis}

\theoremstyle{definition}
\newtheorem{defi}[equation]{Definition}

\DeclareMathOperator*{\colim}{colim}

\definecolor{lightblueforque}{HTML}{8fcaff}

\newtheorem{protoerr}{ERROR}

\newtheorem{protoque}[protoerr]{QUESTION}

\newtheorem{prototbe}[protoerr]{To be expanded}

\newtheorem{protoysr}[protoerr]{You should remember}

\definecolor{colarr2}{rgb}{1.0,0,0}
\definecolor{colarr3}{rgb}{0.5,0,0.5}
\definecolor{colarr4}{rgb}{0,0,1.0}
\definecolor{colarr5}{rgb}{1,0.6,0}
\definecolor{colarr6}{rgb}{0,1.0,0}
\definecolor{colarr7}{rgb}{0.5,0.5,0}
\definecolor{colarr8}{rgb}{0,0.5,0.5}
\definecolor{colarr12}{RGB}{255,102,102}
\definecolor{colarr18}{RGB}{150,39,240}

\definecolor{faintcolarr2}{rgb}{1.0,0.5,0.3}
\definecolor{faintcolarr4}{rgb}{0.5,0.5,1.0}
\definecolor{faintcolarr8}{rgb}{0.2,1.0,1.0}
\definecolor{faintcolarr16}{rgb}{0.2,1.0,0.2}

\definecolor{colnewp}{rgb}{0.4,0,0.7}

\definecolor{colsigma}{rgb}{0.1,1.0,0.3}
\title{A motivic Greenlees spectral sequence towards motivic Hochschild homology.}
\author{Federico Ernesto Mocchetti\footnote{Università degli Studi di Milano - Universit{\"a}t Osnabr{\"u}ck}}
\date{\today}

\begin{document}

\maketitle
\thispagestyle{empty}

\begin{abstract}
    {\footnotesize
    We define a motivic Greenlees spectral sequence by characterising an associated $t$-structure. We then examine a motivic version of topological Hochschild homology for the motivic cohomology spectrum modulo a prime number $p$. Finally, we use the motivic Greenlees spectral sequence to determine the homotopy ring of a related spectrum, given that the base field is algebraically closed with a characteristic that is coprime to $p$. 
    }
\end{abstract}

\setcounter{section}{-1}

\section{Introduction}

In a paper from 2014, \cite{Greenlees2016}, John Greenlees introduced a spectral sequence on (classical) commutative ring spectra as follows.

\begin{prop}[Lemma 3.1, \cite{Greenlees2016}] \label{prop:original.Greenlees}
    If $S\to R \to Q$ is a cofibre sequence of connective commutative algebras augmented over $k$ and $\pi_0(S)=k$, and $R$ is of upward finite type as an $S$-module, then there is a multiplicative spectral sequence:
    \[
        E^2_{s,t}=\pi_{s}(Q)\otimes_k \pi_{t}(S) \Rightarrow \pi_{s+t}(R)
    \] 
    with differentials:
    \[
        d^r: E^r_{s,t} \to E^r_{s-r,t+r-1}.
    \]
\end{prop}

The goal of the \hyperref[sec:motivic.spe.seq]{first section of this paper} is to produce an analogue of this spectral sequence in the context of the stable motivic homotopy category $\mathcal{SH}(S)$.  More precisely, we are interested in studying the truncation of those spectra that carry a $Q$-module structure, where $Q$ is a commutative algebra $Q \in CAlg(\mathcal{SH}(S))$. We begin by constructing a suitable $t$-structure; 
in particular, we make use of \cite[Proposition 1.4.4.11]{Lurie2017} to generate the non-negative part $Mod_{Q,\geq 0}$ from the small collection of objects $\{ \Sigma^{0,i} Q \}_{i \in \mathbb{Z}}$. This reflects our desire to control the truncations of a certain object by imposing conditions on its homotopy groups above or below a certain total degree. In the \hyperref[subs:t.struct]{first subsection}, we prove results in this sense (see, for instance, equations \ref{eqn:homotopy.of.fibre1} and \ref{eqn:homotopy.of.fibre2} or corollary \ref{cor:truncation.of.shifts}). Subsequently, we prove that this $t$-structure is compatible with the monoidal structure on $Mod_Q^{\otimes}$ (proposition \ref{prop:prod.on.t.str}).

In the \hyperref[subs:spectr.seq]{second subsection}, we introduce the motivic analogue of proposition \ref{prop:original.Greenlees} (proposition \ref{prop:ugly.spectral.sequence}), as a tri-graded upper half plane spectral sequence, starting from the $E^2$ page and with differentials of the form:
\[
    d^r:E^r_{s,t,*} \to E^r_{s-r,t+r-1,*}
\]
Observe that the third degree (which will be related to the weight of the spectra involved) is not altered by the differentials so that we might conceive this tri-graded spectral sequence as an infinite family of bi-graded ones. The remaining part of the subsection is dedicated to results that help to improve the appearance of the $E^2$ page and the convergence term.

In the \hyperref[sec:MHH.tau-1]{second section}, we focus on the $\tau$-inverted version of motivic Hochschild homology. Motivic Hochschild homology can be considered the immediate analogue of topological Hochschild homology in the stable motivic homotopy category. In fact, given $Q\in \mathcal{SH}(S)$ a ring spectrum, we define the motivic Hochschild homology $MHH(Q)$ of $Q$ as the derived tensor product:
\[
    Q \wedge_{Q \wedge Q^{op}}Q.
\]
In the event $Q$ is $E_{\infty}$, one can equivalently express this as a geometric realisation along the simplicial circle:
\[
MHH(Q) \cong S^{1}_s \otimes Q.
\]
The purpose of this second section is to compute the homotopy ring of:
\[
    MHH(M\mathbb{Z}/p)[\tau^{-1}]= M\mathbb{Z}/p[\tau^{-1}]\wedge_{M\mathbb{Z}/p\wedge M\mathbb{Z}/p[\tau^{-1}]} M\mathbb{Z}/p[\tau^{-1}]
\]
where $M\mathbb{Z}/p$ is the Suslin-Voevodsky mod-$p$ motivic cohomology ring spectra for $p$ any prime number and the basis $S=Spec(F)$ is the spectrum of an algebraically closed field of characteristic different from $p$; $\tau$ is a canonical class in $\pi_{0,-1}M\mathbb{Z}/p$. At first, we recall how one makes homotopy elements invertible at the level of spectra; we then verify that the object that one obtains satisfies all the conditions to get a properly-looking, first-quadrant spectral sequence (proposition \ref{prop:spectral.seq.tau.-1}).

\hyperref[subs:computation]{The last subsection} is devoted to the computation of the homotopy of $MHH(M\mathbb{Z}/p)[\tau^{-1}]$. After an auxiliary result, we prove that for any prime number $p$, one has:
\[
    \pi_{*,*} MHH(M\mathbb{Z}/p)[\tau^{\pm 1}] \cong \mathbb{F}_p[\mu_0]
\]
with $|\mu_0|=(2,0)$ (propositions \ref{prop:pi.MHH(MZ/2)[tau-1]} and \ref{prop:MHH(p).[tau-1]}).

\vspace{1cm}

I sincerely thank my advisors, Paul Arne {\O}stv{\ae}r and Markus Spitzweck, for challenging me with this problem and providing continuous guidance while realising this manuscript. I am also grateful to Bj{\o}rn Dundas for his invaluable help with the setup of the spectral sequence for $MHH(M\mathbb{Z}/p)[\tau^{-1}]$.

\clearpage

\tableofcontents
\vspace{1.5cm}

\section{A motivic spectral sequence}\label{sec:motivic.spe.seq}

In this section, we construct a $t$-structure on motivic spectra and prove related properties. We then associate a homotopy spectral sequence to such $t$-structure. Finally, we prove some results that refine the aspect of the $E^2$ page and of the convergence term of the spectral sequence

\subsection{The \texorpdfstring{$t$-}{t-}structure}\label{subs:t.struct}

In this section, we introduce the $t$-structure at the basis of the spectral sequence appearing in proposition \ref{prop:ugly.spectral.sequence}. In particular, we focus on results providing a nicer description of truncations and fibres from the point of view of homotopy groups.

We begin with a general characterisation of the subcategories of truncated objects in an $\infty$-category with a $t$-structure in terms of the mapping spaces.

\begin{lemma}\label{lemma:testongen}
Let $\{ C_i \}_{i \in \mathcal{I}} \subseteq \mathcal{C}$ be a small collection of objects in a presentable stable $\infty$-category $\mathcal{C}$. Consider the $t$-structure generated by them under colimits and extensions \cite[Proposition 1.4.4.11]{Lurie2017}. Then:
\begin{enumerate}
    \item $Y \in \mathcal{C}_{<0}$ if and only if $map_{\mathcal{C}}(X,Y)$ is contractible for all $X \in \mathcal{C}_{\geq 0}$.
    \item $Y \in \mathcal{C}_{<0}$ if and only if $map_{\mathcal{C}}(C_i,Y)$ is contractible for all ${i \in \mathcal{I}}$.
\end{enumerate}
\end{lemma}

\begin{proof}
    Recall that, by the definition of a $t$-structure on a stable infinity category $\mathcal{C}$ \cite[Definition 1.2.1.4]{Lurie2017}, an object $Y \in \mathcal{C}_{<0}$ if $\pi_0(map_{\mathcal{C}}(X,Y))=Hom_{h\mathcal{C}}(X,Y)=0$ for all $X \in (h\mathcal{C})_{\geq 0}$.
    
\begin{enumerate}
    \item From left to right, first observe that the (simplicial) suspension of some $X \in \mathcal{C}$ is obtained as a pushout:
    \[
        \begin{tikzcd}
            X \arrow[r] \arrow[d] & \ast \arrow[d] \\
            \ast \arrow[r] & \Sigma^{1,0} X.
        \end{tikzcd}
    \]

    In particular, all positive simplicial suspensions $X[j]:=\Sigma^{j,0} X$ of any $X \in \mathcal{C}_{\geq 0}$ belong to $\mathcal{C}_{\geq 0}$. But then, if $Y \in \mathcal{C}_{<0}$, $\pi_j \, map_{\mathcal{C}}(X,Y) \simeq \pi_0 \, map_{\mathcal{C}}(X[j],Y)\cong 0$ \cite[Notation 1.1.2.17]{Lurie2017}. So $map_{\mathcal{C}}(X,Y)$ is contractible, as required. 

    The other implication is obvious.

    \item We shall use the characterisation from the previous point in this part of the proof. 
    
    The implication from left to right is immediate.

    From right to left, let $ Y \in \mathcal{C}$ be such that $map_{\mathcal{C}}(C_i,Y)\simeq \ast$ for all $i \in \mathcal{I}$.
    Consider the collection:
    \[
    \mathcal{C}'=\mathcal{C}'_Y=\{ X \in \mathcal{C} \text{ such that } map_{\mathcal{C}}(X,Y) \simeq \ast\}
    \]
    We want to show that $Y \in \mathcal{C}_{<0}$, in other words, that $\mathcal{C}_{\geq 0} \subseteq \mathcal{C}'$. By hypothesis, the generators $\{C_i\}$ belong to $\mathcal{C}'$. If we show that $\mathcal{C}'$ is closed under colimits and extensions, we are done. 
    
    Let us begin with colimits. This follows from the usual functorial interaction of colimits and mapping spaces. In fact, suppose that we have a diagram $J \to \mathcal{D}$ in some $\infty$-category $\mathcal{D}$ that admits a colimit. We want to show that, for all $B \in \mathcal{D}$:
    \[
        map_{\mathcal{D}} (\colim_J A_j, B) \cong \lim_J map_{\mathcal{D}} (A_j, B)
    \]
    To show this, first notice that a colimit diagram in the $\infty$-category $\mathcal{D}$ corresponds to a limit diagram in the opposite category $\mathcal{D}^{op}$; moreover, one has:
    \[
        map_{\mathcal{D}} (\colim_{J \to \mathcal{D}} A_j, B) \cong map_{\mathcal{D}^{op}} (B, \lim_{J^{op} \to \mathcal{D}^{op}} A_j)
    \]
    Now pass to presheaves:
   \[ \begin{tikzcd}
       J^{op} \arrow[d]&&\\
       \mathcal{D}^{op} \arrow[r, "\tilde{\mathcal{Y}}"]&
       \tilde{\mathscr{P}}(\mathcal{D})=Fun(\mathcal{D}, \mathcal{S}) \arrow[r, "ev_B"]&
       \mathcal{S} \\
       A_j \arrow[r, mapsto] & \tilde{\mathcal{Y}}(A_j) \arrow[r, mapsto] &
       \tilde{\mathcal{Y}}(A_j)(B).
    \end{tikzcd}\]
    Here we indicate by $\tilde{\mathcal{Y}}$ the Yoneda embedding of $\mathcal{D}^{op}$ \cite[Section 5.1.3]{Lurie2009} and \cite[Section 4.2]{Land2021}. First observe that $\tilde{\mathcal{Y}}(A_j)(B) \cong map_{\mathcal{D}^{op}}(B,A_j)\cong map_{\mathcal{D}}(A_j,B)$ by \cite[Corollary 4.2.8]{Land2021}. Now, the Yoneda embedding sends limits in $\mathcal{D}^{op}$ to limits in presheaves by \cite[ Proposition 5.1.3.2]{Lurie2009}. The evaluation at $B$ sends a limit diagram in $Fun(\mathcal{D}, \mathcal{S})$ to a limit diagram in $\mathcal{S}$ \cite[Proposition 6.2.10]{Cisinski2019}.
    So a $J$-indexed colimit in $\mathcal{D}$ is sent to a $J$-indexed limit in spaces, proving our assertion.
    
    Consider now a fibre sequence $X' \to X \to X''$ with $X',\,X'' \in \mathcal{C}'$. Dually, this induces a (co)fibre sequence:
    \[
        map_{\mathcal{C}}(X'',Y) \to map_{\mathcal{C}}(X,Y) \to map_{\mathcal{C}}(X',Y)
    \]

    Consider the associated long exact sequence on homotopy groups:
    \begin{multline*}
        \cdots \to \pi_{n+1}map_{\mathcal{C}}(X',Y)\to \pi_{n}map_{\mathcal{C}}(X'',Y) \to \pi_{n}map_{\mathcal{C}}(X,Y)\to \\
        \pi_{n}map_{\mathcal{C}}(X',Y)\to \pi_{n-1}map_{\mathcal{C}}(X'',Y) \to \cdots
    \end{multline*}
    Since by hypothesis all the homotopy groups of $map_{\mathcal{C}}(X',Y)$ and $map_{\mathcal{C}}(X'',Y)$ vanish, so must those of $map_{\mathcal{C}}(X,Y)$. We conclude that $map_{\mathcal{C}}(X,Y)$ is contractible, again by Whitehead for Kan complexes \cite[\href{https://kerodon.net/tag/00WV}{Tag 00WV}]{Lurie2018} (or equivalently the fact that the $\infty$-topos of spaces is hypercomplete).
    
\end{enumerate}
\end{proof}

We are ready to introduce the $t$-structure that gives rise to the spectral sequence studied in this paper.

\begin{defi}\label{defi:t-structure}
    Let $Mod_Q$ be the $\infty$-category of modules over a spectrum $Q \in CAlg(\mathcal{SH}(S))$; consider the collection of $\{ \Sigma^{0,i}Q\}_{i \in \mathbb{Z}}\subseteq Mod_Q$. As $Mod_Q$ is presentable, we can apply \cite[Proposition 1.4.4.11]{Lurie2017}: we define $Mod_{Q,\geq 0}$ to be the smallest full subcategory of $Mod_Q$ that contains the $\{ \Sigma^{0,i}Q\}_{i \in \mathbb{Z}}$ and is closed under colimits and extensions.
\end{defi}

We would like to characterise the truncations of $Mod_Q$ with respect to this $t$-structure in terms of the homotopy groups of their elements. As we will see, this requires making appropriate assumptions on the homotopy groups of $Q$. 

From \cite[Proposition 4.6.2.17]{Lurie2017} (here we make use of the identifications of \cite[Corollary 4.5.1.6]{Lurie2017}) the map $\mathbf{1} \to Q$  in $SH(S)$ induces an adjunction:
\begin{equation}\label{eq:adj.SH.ModQ}
    \begin{tikzcd}
        \mathcal{SH}(S) \arrow[rr, bend left=30, "-\wedge Q"] & \bot &
        Mod_{Q} \arrow[ll, bend left=30, "U"]
    \end{tikzcd}
\end{equation}
where $U$ is the forgetful functor.
Then, by \cite[Remark 5.1.4]{Land2021}, we have equivalences:
\begin{equation} \label{eqn:sh.modQ.adjunction}
    map_{Mod_{Q}}(\Sigma^{i,j}Q,Z) \simeq map_{\mathcal{SH}(S)}(S^{i,j},Z).
\end{equation}

\begin{rmk}\label{rmk:easyprop:htpygpscof}
    It is immediate from the definition of the $t$-structure that for all $Y \in Mod_{Q,\leq n}$, all $i \geq n+1$ and all $j$, by \ref{eqn:sh.modQ.adjunction}:
\[
    map_{\mathcal{SH}(S)}(S^{i,j},Y)\simeq map_{Mod_{Q}}(\Sigma^{i,j}Q,Y)\cong \star,
\]
without any further assumption on $Q$. In particular, the homotopy groups $\pi_{i,j}(Y)=\pi_0 map_{\mathcal{SH}(S)}(S^{i,j},Y)$ vanish for $i \geq n+1$ and all $j$.
\end{rmk}

\begin{prop} \label{prop:htpygpscof}
    Suppose that $\pi_{i,j}(Q)=0$ for $i<0$ and let $Y \in Mod_{Q,\geq n}$.
    Then we have $\pi_{i,j}(Y)=0$ for $i < n$ and all $j$.
\end{prop}

\begin{proof} 
    We may assume $n=0$, as the general case is just a translation of the argument below.

    First, observe that $Mod_{Q,\geq 0}$ consists of cellular modules \cite[Definition 2.1]{DugIsa2005}; recall that in particular a module is called cellular if it belongs to the smallest subcategory of $Mod_Q$ containing $\{\Sigma^{a,b}Q\}_{a,\, b \in \mathbb{Z}}$ and closed under colimits. In fact, let $\mathcal{M} \subseteq Mod_{Q,\geq 0}$ be the full subcategory spanned by the cellular modules. Then:
    \begin{itemize}
        \item The generating modules $\{\Sigma^{a,b}Q\}_{a \geq 0,\, b \in \mathbb{Z}}\in \mathcal{M}$
        \item $\mathcal{M}$ is closed under colimits because both $Mod_{Q,\geq 0}$ and cellular modules are
        \item $\mathcal{M}$ is closed under extensions because given a cofibre sequence:
        \[
            X \to Y \to Z
        \]
        with $X$ and $Z$ in $\mathcal{M}$, then $Z \in Mod_{Q,\geq 0}$ because the non-negative part of a $t$-structure is closed under extensions, and $Y$ is cellular because of \cite[Lemma 2.5]{DugIsa2005}. So $Y \in \mathcal{M}$.
    \end{itemize}
    As $Mod_{Q,\geq 0}$ is the smallest full subcategory of $Mod_Q$ that contains $\{\Sigma^{a,b}Q\}_{a \geq 0,\, b \in \mathbb{Z}}$ and is closed under colimits and extensions, $\mathcal{M} = Mod_{Q,\geq 0}$.

    Cellularity implies very nice features; for example, homotopy groups detect equivalences \cite[Corollary 7.2]{DugIsa2005}, and every cellular module is a (possibly infinite) direct sum of sphere modules $\{\Sigma^{a,b}Q\}_{a,\, b \in \mathbb{Z}}$ \cite[Remark 7.4]{DugIsa2005}. The proof of \cite[Proposition 7.3]{DugIsa2005} allows us to be more precise: every cellular module is a direct sum of sphere modules $\{\Sigma^{a,b}Q\}_{a,\, b \in \mathbb{Z}}$, which correspond to the non-trivial elements in the homotopy ring. This implies that, in the context of this proposition, every object in $Mod_{Q,\geq 0}$ is actually a direct sum of modules in $\{\Sigma^{a,b}Q\}_{a \geq 0,\, b \in \mathbb{Z}}$.

    Now, all the elements in $\{\Sigma^{a,b}Q\}_{a \geq 0,\, b \in \mathbb{Z}}$ have trivial homotopy in negative degree:
    \[
    \pi_{i,j}(\Sigma^{a,b}Q)=[S^{i,j},\Sigma^{a,b}Q]\cong [S^{i-a,j-b},Q]=\pi_{i-a,j-b}(Q)\cong 0
    \] 
    for $i \leq 0$, by hypothesis. By \cite[Proposition 9.3]{DugIsa2005}, given a directed (in particular, discrete) system  $\alpha \to E_{\alpha}$, we have:
    \[
        \colim_{\alpha} \pi_{i,j}(E_{\alpha}) \cong \pi_{i,j}(\colim_{\alpha} E_{\alpha}),
    \]
    where the colimit of the $E_{\alpha}$ has to be intended in an $\infty$-categorical (homotopical) sense. But then any element of $Mod_{Q, \geq 0}$ has trivial homotopy in negative degrees.
    
\end{proof}

Let's put together the above results. Let $Y \in Mod_Q$ be a cellular module. Observe that we have a diagram of fibre sequences \cite[Remark 1.2.1.8]{Lurie2017}:
\[
\begin{tikzcd}
    \tau_{\geq n+1} Y \arrow[r] \arrow[d, "f"] &
    Y \arrow[r] \arrow[d]&
    \tau_{\leq n} Y \arrow[r] \arrow[d, "g"]&
    \tau_{\geq n+1} {Y\left[1 \right]} \arrow[r] \arrow[d, "{f\left [1\right ]}"] &
    {Y\left [1\right ]} \arrow[r] \arrow[d]&
    \ldots \\    
    \tau_{\geq n} Y \arrow[r] \arrow[d] &
    Y \arrow[r] \arrow[d]&
    \tau_{\leq n-1} Y \arrow[r] \arrow[d]&
    \tau_{\geq n} Y[1] \arrow[r] \arrow[d] &
    {Y \left [1\right ]} \arrow[r] \arrow[d]&
    \ldots \\
    cofib(f) \arrow[r] \arrow[d] &
    0 \arrow[r]  \arrow[d] &
    cofib(g) \arrow[r] \arrow[d]  &
    {cofib(f)\left [1\right ]} \arrow[r]  \arrow[d] &
    0 \arrow[r] \arrow[d] &
    \ldots\\
    \tau_{\geq n+1} Y[1] \arrow[r] &
    Y[1] \arrow[r] &
    \tau_{\leq n} Y[1] \arrow[r] &
    \tau_{\geq n+1} {Y\left[2 \right]} \arrow[r] &
    {Y\left [2\right ]} \arrow[r] &
    \ldots 
\end{tikzcd}
\]
For our spectral sequence, we are interested in characterising the (co)fibre of the map $g$. From the diagram, we can deduce:

\begin{enumerate}
    \item From the third row, $cofib(g) \cong cofib(f)[1]$, so $fib(g) \cong cofib(f)$.
    
    \item From the third column, we can extract the cofibre sequence:
    \[
        \tau_{\leq n-1} Y \to cofib(g) \to \tau_{\leq n} Y[1]
    \]
    which produces a long exact sequence of homotopy groups:

    \begin{multline*}
        \pi_{i,j}(\tau_{\leq n-1} Y) \to \pi_{i,j}(cofib(g)) \to \pi_{i-1,j}(\tau_{\leq n} Y) \to \\
        \pi_{i-1,j}(\tau_{\leq n-1} Y) \to \pi_{i-1,j}(cofib(g)) \to \pi_{i-2,j}(\tau_{\leq n} Y)\to \ldots
    \end{multline*}

    From Remark \ref{rmk:easyprop:htpygpscof}, for all $j \in \mathbb{Z}$:
    \begin{align*}
        &\pi_{i,j}(\tau_{\leq n-1} Y)\cong 0 \text{ for }i\geq n \\
        &\pi_{i-1,j}(\tau_{\leq n} Y) \cong 0 \text{ for } i-1 \geq n+1 \text{, or } i \geq n+2.
    \end{align*}
    
    So we have exact sequences, for all $j \in \mathbb{Z}$:
    \begin{align*}
        0 \to \pi_{i, j}(cofib(g)) &\to 0 \qquad \text{ for } i\geq n+2 \\
        0 \to \pi_{n+1,j}(cofib(g)) &\to \pi_{n,j}(\tau_{\leq n} Y) \to 0.
    \end{align*}

    From which we get the isomorphisms: $\pi_{i, j}(cofib(g)) \cong 0$ for  $i\geq n+2$ and $\pi_{n+1,j}(cofib(g)) \cong \pi_{n,j}(\tau_{\leq n} Y)$. Given the isomorphism in the previous point, this becomes:
    \begin{equation} \label{eqn:homotopyidentities0}
        \begin{aligned} 
            \pi_{i, j}(cofib(f)) \cong 0 \text{ for }  i\geq n+1\\
            \pi_{n,j}(cofib(f)) \cong \pi_{n,j}(\tau_{\leq n} Y)
        \end{aligned}
    \end{equation}

    \item From the horizontal cofibre sequence
    \[
    \tau_{\leq n} Y \to \tau_{\geq n+1} {Y\left[1 \right]} \to Y[1] \to \tau_{\leq n} Y[1] \to\ldots
    \]
    we get an exact sequence of homotopy groups:
    \[
    \pi_{i,j}(\tau_{\leq n} Y) \to \pi_{i-1,j}(\tau_{\geq n+1} {Y})\to \pi_{i-1,j}(Y) \to \pi_{i-1,j}(\tau_{\leq n} Y) \to \ldots 
    \]
    Applying the vanishing results of remark \ref{rmk:easyprop:htpygpscof} we get exact sequences for all $j\in \mathbb{Z}$:
    \begin{equation*}
        0\to \pi_{i-1,j}(\tau_{\geq n+1} {Y})\to \pi_{i-1,j}(Y) \to 0 \text{ for } i \geq n+2
    \end{equation*}
    So we have isomorphisms for all $j \in \mathbb Z$:
    \begin{equation}  \label{eqn:homotopyidentities1}
    \pi_{i,j}(\tau_{\geq n} {Y})\cong \pi_{i,j}(Y) \text{ for } i \geq n.
    \end{equation}
\end{enumerate}

If we further assume that $\pi_{i,j}(Q)=0$ for $i<0$ and all $j$, we can apply proposition \ref{prop:htpygpscof}, and get:

\begin{enumerate}
    \item From the first column, since $\tau_{\geq n} Y \in Mod_{Q, \geq n}$ and $\tau_{\geq n+1} Y[1] \in Mod_{Q, \geq n+1} \subseteq Mod_{Q, \geq n}$, $cofib(f) \in Mod_{Q, \geq n}$ from \cite[Proposition 1.2.1.16]{Lurie2017}; then, by  proposition \ref{prop:htpygpscof}: 
    \begin{equation} \label{eqn:homotopyidentities2}
        \pi_{i,j}(cofib(f))=0 \text{ for } i<n
    \end{equation} 
    and all $j$.
    \item From the horizontal cofibre sequence:
     \[
    \tau_{\geq n+1} Y \to Y \to \tau_{\leq n} Y \to \tau_{\geq n+1} {Y\left[1 \right]} \to\ldots
    \]
    we get an exact sequence of homotopy groups:
    \[
    \pi_{i,j}(\tau_{\geq n+1} Y ) \to \pi_{i,j}(Y) \to \pi_{i,j}(\tau_{\leq n} Y) \to \pi_{i-1,j}(\tau_{\geq n+1} {Y})\to \ldots 
    \]
    Applying this time the vanishing results of proposition \ref{prop:htpygpscof} we get exact sequences for all $j\in \mathbb{Z}$:
    \begin{equation*} 
        0 \to \pi_{i,j}(Y) \to \pi_{i,j}(\tau_{\leq n} Y) \to 0 \text{ for } i \leq n 
    \end{equation*}
    So we have isomorphisms for all $j \in \mathbb Z$:
    \begin{equation} \label{eqn:homotopyidentities3}
      \pi_{i,j}(Y) \cong \pi_{i,j}(\tau_{\leq n} Y)  \text{ for } i \leq n 
    \end{equation}
\end{enumerate}

Putting all together, $cofib(f) \cong fib(g)$ is a $Q$-module, with homotopy groups, for all $j$:
\begin{equation}\label{eqn:homotopy.of.fibre1}
\begin{aligned}
     \pi_{i, j}(cofib(f)) \cong 0 \text{ for }  i\geq n+1\\
     \pi_{n,j}(cofib(f)) \cong \pi_{n,j}(Y)
\end{aligned}
\end{equation}
from \ref{eqn:homotopyidentities0} and \ref{eqn:homotopyidentities1}. If one further assumes $\pi_{i,j}(Q)=0$ for $i<0$ and all $j$, one deduces also, for all $j$:
\begin{gather}\label{eqn:homotopy.of.fibre2}
     \pi_{i, j}(cofib(f)) \cong 0 \text{ for }  i \leq n-1
\end{gather}
from \ref{eqn:homotopyidentities2}.

The above results can in fact be related by something happening to the structure of the truncated objects at a deeper level.

\begin{lemma} \label{lemma:heart}
    The homotopy groups $\pi_{i,j}(Q)$ vanish for $i > 0$ and all $j$ if and only if the generating set $\{ \Sigma^{0,i}Q\}_{i \in \mathbb{Z}}$ lies in the hearth of the $t$-structure.
\end{lemma}

\begin{proof}
    We just have to verify that $\{ \Sigma^{0,i}Q\}_{i \in \mathbb{Z}} \subseteq Mod_{Q,<1}$, as the inclusion $\{ \Sigma^{0,i}Q\}_{i \in \mathbb{Z}} \subseteq Mod_{Q,\geq 0}$ holds by definition.
    
    Given any $Z \in Mod_Q$, by the above lemma \ref{lemma:testongen}, we can reduce to test if $Z \in Mod_{Q,<1}$ by checking the contractibility of the mapping spaces from generators of $Mod_{Q,\geq 0}$ to $Z$. For our $t$-structure, this means:
    \[
        map_{Mod_Q}(\Sigma^{1,j}Q, Z) \simeq \ast.
    \]
    If we now choose $Z= \Sigma^{0,i}Q$ in the generating set, we have, by \ref{eqn:sh.modQ.adjunction}:
     \[
        map_{Mod_Q}(\Sigma^{1,j}Q, \Sigma^{0,i}Q) \simeq map_{\mathcal{SH}(S)}(S^{1,j},\Sigma^{0,i}Q) \simeq map_{\mathcal{SH}(S)}(S^{1,j-i},Q)
     \]
     Let's consider the homotopy groups of this mapping space. By \cite[Notation 1.1.2.17]{Lurie2017}, we have an identification for all non-negative $n$:
     \[
        \pi_n map_{\mathcal{SH}(S)}(S^{1,j-i},Q) \cong [S^{1+n,j-i},Q]=\pi_{1+n,j-i}(Q).
     \]
     Hence, the assumption on the homotopy groups of $Q$ is equivalent to the contractibility of the various mapping spaces.
\end{proof}

This proof generalises to:
\begin{cor}\label{cor:truncation.of.shifts}
    If the homotopy groups $\pi_{i,j}(Q)$ vanish for $i > 0$ and all $j$, then:
    \[
    \tau_{\leq n}(\Sigma^{p,q} Q) \cong 
    \begin{cases}
        \Sigma^{p,q} Q \text{ if } n \geq p \\
        * \text{ if } n < p.
    \end{cases}
    \]
\end{cor}
Suppose that $Q$ lies in the heart of our $t$-structure; observe that we have equivalences: 
    \[
     Q \cong \tau_{\leq 0} Q \cong fib( \tau_{\leq 0} Q \to \tau_{\leq -1} Q).
    \]

Suppose in this setting that $Y$ is cellular, in other words, a coproduct of shifted copies of $Q$:
\[
    Y= \bigvee_{\alpha} \Sigma^{p_{\alpha},q_{\alpha}}Q.
\] 

Now, recall that the truncation functors $\tau_{\leq n}: Mod_Q \to Mod_{Q,\leq n}$ are left adjoint to the inclusion $Mod_{Q,\leq n} \subseteq Mod_Q$ \cite[Proposition 1.2.1.5.]{Lurie2017};
then, by the adjoint functor theorem \cite[Corollary 5.5.2.9]{Lurie2009}, they preserve small colimits.

So:
\[
\tau_{\leq n} Y =\tau_{\leq n} \left (\bigvee_{\alpha}\Sigma^{p_{\alpha},q_{\alpha}}Q \right)=
\bigvee_{\alpha}\tau_{\leq n}\left ( \Sigma^{p_{\alpha},q_{\alpha}}Q \right ).
\]
From \ref{cor:truncation.of.shifts}, we then have:
\begin{equation} \label{eqn:taun.Y.if.union}
    \tau_{\leq n} Y=\bigvee_{\alpha \text{ s.t. } p_{\alpha}\leq n}\Sigma^{p_{\alpha},q_{\alpha}}Q. 
\end{equation}

In particular, it follows that:
    \begin{equation}
        fib(\tau_{\leq n} Y \to \tau_{\leq n-1} Y) = \bigvee_{\alpha \text{ s.t. } p_{\alpha} = n} \Sigma^{n,q_{\alpha}}Q
    \end{equation}

\begin{rmk}
    Here we are actually adopting the common practice of calling $\tau_{\leq n}$ the composite $inc_n \circ \tau_{\leq n}: Mod_Q \to Mod_Q$, $inc_n$ being the inclusion $inc_n: \tau_{\leq n}Mod_Q \to Mod_Q$, so that it is possible to compare the different truncations $\tau_{\leq n} Y$ and $\tau_{\leq n-1} Y$. 
\end{rmk}

This allows, in particular, to recover \ref{eqn:homotopy.of.fibre1}, and once we assume that also $\pi_{i,*}Q=0$ for $i < 0$, we get \ref{eqn:homotopy.of.fibre2} as well. Nonetheless, it is important to note that this result would be useful even if $\pi_{*,*}Q$ were nonzero in some negative degrees: while we wouldn't achieve the same identifications between the homotopy of the fibres and that of $Y$, the fibres would still exhibit fairly reasonable and computable homotopy groups.

To conclude this part, we observe that our filtration is compatible with the monoidal structure on modules. In fact, \cite[Theorem 4.5.2.1]{Lurie2017} and \cite[Theorem 4.5.3.1]{Lurie2017} allow us to promote the free/forgetful adjunction \ref{eq:adj.SH.ModQ} to a symmetric monoidal one:

\begin{equation}\label{eq:smon.adj.SH.ModQ}
    \begin{tikzcd}
        \mathcal{SH}(S)^{\wedge} \arrow[rr, bend left=30, "-\wedge Q"] & \bot &
        Mod_{Q}^{\otimes_Q} \arrow[ll, bend left=30, "U"]
    \end{tikzcd}
\end{equation}
where on the left-hand side we put the usual coproduct monoidal structure and on the right-hand side we have the product induced from the bimodules structures \cite[Proposition 4.4.3.12.]{Lurie2017}.

Moreover, as $\mathcal{SH}(S)^{\wedge}$ is a presentable symmetric monoidal category, so is $Mod_{Q}^{\otimes_Q}$, by \cite[Theorem 3.4.4.2]{Lurie2017}. In particular, the product $- \otimes_Q -$ preserves colimits independently in each variable.

\begin{prop}\label{prop:prod.on.t.str}
    The $t$-structure is compatible with the product on $Mod_Q^{\otimes_Q}$, in the sense that if $X \in Mod_{Q,\geq n}$ and $Y \in Mod_{Q,\geq m}$, then $X \otimes_Q Y \in Mod_{Q,\geq n+m}$.
\end{prop}

\begin{proof}
    As observed in \cite[Remark 2.2.1.4]{Lurie2017}, it is enough to prove the thesis for $n=m=0$. Moreover, as the product is symmetric, we can reduce to showing it for the first variable only.
    
    Let then $Y \in Mod_{Q,\geq 0}$ and define $\mathcal{C}$ to be the full subcategory of $Mod_Q$ spanned by those $X$ such that $X \otimes_Q Y \in Mod_{Q,\geq 0}$. We wish to prove that $Mod_{Q,\geq 0} \subseteq \mathcal{C}$. By definition of $Mod_{Q,\geq 0}$, it is enough to show that $\mathcal{C}$ contains the generators $\{\Sigma^{0,i}Q\}_{i \in \mathbb{Z}}$ and is closed under small colimits and extensions.

    Let then $i$ be any integer; we want to show that for any $Y \in Mod_{Q,\geq 0}$, $\Sigma^{0,i}Q\otimes_Q Y \in Mod_{Q,\geq 0}$. We use a similar argument: let $\mathcal{D}$ be the full subcategory of $Mod_Q$ spanned by those $Y$ such that $\Sigma^{0,i}Q\otimes_Q Y \in Mod_{Q,\geq 0}$; we wish to prove that $Mod_{Q,\geq 0} \subseteq \mathcal{D}$. 

    First of all, $\{\Sigma^{0,j}Q\}_{j \in \mathbb{Z}}\subseteq \mathcal{D}$. In fact, due to monoidality of the functor $-\wedge Q:  \mathcal{SH}(S)^{\wedge} \to Mod_{Q}^{\otimes_Q}$:
    \[
    \begin{aligned}
        \Sigma^{0,i}Q \otimes_Q \Sigma^{0,j}Q &= (S^{0,i} \wedge Q) \otimes_Q (S^{0,j} \wedge Q) \cong (S^{0,i} \wedge S^{0,j}) \wedge Q \\ 
        &\cong S^{0,i+j} \wedge Q = \Sigma^{0,i+j}Q \in Mod_{Q,\geq 0}
    \end{aligned}
    \]
    Then, $\mathcal{D}$ is closed under colimits, as the tensor product is compatible with them; suppose in fact $Y_{\alpha} \in \mathcal{D}$ for all $\alpha \in A$ some indexing diagram:
    \[
        \Sigma^{0,i}Q \otimes_Q (\colim_{\alpha \in A} Y_{\alpha}) \cong \colim_{\alpha \in A} (\Sigma^{0,i}Q \otimes_Q   Y_{\alpha})
    \]
    As $Y_{\alpha} \in \mathcal{D}$, $\Sigma^{0,i}Q \otimes_Q   Y_{\alpha} \in Mod_{Q,\geq 0}$. But $Mod_{Q,\geq 0}$ is closed under colimits, hence the claim.
    
    Finally, extensions: we must show that for every fibre sequence $Y' \to Y \to Y''$, with $Y' \in \mathcal{D}$ and $Y'' \in \mathcal{D}$, also $Y \in \mathcal{D}$. Now a fibre sequence in a stable $\infty$ category is the same as a cofibre sequence, hence it is defined by a colimit (pushout):
    \[
        \begin{tikzcd}
            Y' 
                \arrow{r}
                \arrow{d}
            & Y
                \arrow{d}
            \\ 0
                \arrow{r}
            & Y''
        \end{tikzcd}
    \]
    Since $\Sigma^{0,i}Q \otimes_Q -$ preserves colimits, also the induced diagram:
    \[
        \begin{tikzcd}
            \Sigma^{0,i}Q \otimes_Q Y' 
                \arrow{r}
                \arrow{d}
            & \Sigma^{0,i}Q \otimes_Q Y
                \arrow{d}
            \\ \Sigma^{0,i}Q \otimes_Q 0
                \arrow{r}
            & \Sigma^{0,i}Q \otimes_Q Y''
        \end{tikzcd}
    \]
    gives rise to a cofibre sequence
    \[
    \Sigma^{0,i}Q \otimes_Q Y' \to \Sigma^{0,i}Q \otimes_Q Y \to \Sigma^{0,i}Q \otimes_Q Y''
    \]
    As $Y' \in \mathcal{D}$ and $Y'' \in \mathcal{D}$, $\Sigma^{0,i}Q \otimes_Q Y' \in Mod_{Q,\geq 0}$ and $\Sigma^{0,i}Q \otimes_Q Y'' \in Mod_{Q,\geq 0}$. But by construction $Mod_{Q,\geq 0}$ is closed under extensions, so $\Sigma^{0,i}Q \otimes_Q Y \in Mod_{Q,\geq 0}$.

    Thus $Mod_{Q,\geq 0} \subseteq \mathcal{D}$, and hence $\mathcal{C}$ contains the generators $\{\Sigma^{0,i}Q\}_{i \in \mathbb{Z}}$. One then proves that it is closed under colimits and extensions by the same argument we used for $\mathcal{D}$, by fixing a generic element $Y$ of $Mod_{Q,\geq 0}$ and considering the image under $ - \otimes_{Q}Y$ of colimit and extension diagrams.
\end{proof}

This in particular implies that connective part $Mod_{Q,\geq 0}$ inherits a symmetric monoidal structure such that the truncation functors:
\[
    \tau_{\leq n}: Mod_{Q,\geq 0} \to Mod_{Q,\geq 0}
\]
are symmetric monoidal \cite[Example 2.2.1.10]{Lurie2017}.
Hence they preserve algebra objects, so, if $R \in Alg_{E_{\infty}}(Mod_{Q, \geq 0})$ is a connective commutative algebra, then the whole filtration:
\[
    R \to \ldots \to R_{\leq n} \to R_{\leq n-1} \to \ldots \to R_{\leq 1} \to R_{\leq 0}
\]
is made of $E_{\infty}$ motivic ring spectra over $Q$ (recall the identification $Alg_{E_{\infty}}(Mod_Q) \cong Alg_{E_{\infty}}(SH(S))_{/Q}$ from \cite[section 3.4.1]{Lurie2017}).

\clearpage

\clearpage
\subsection{The spectral sequence}\label{subs:spectr.seq}

We start with some basic assumptions. All objects are assumed to be cellular.
\begin{hypothesis} \label{hyp:map.of.spectra}
    Let $Q$ be a commutative motivic ring spectrum in $\mathcal{SH}(S)$. Suppose we are given a map of commutative motivic ring spectra in $R \to Q$ in $Alg_{E_{\infty}}(\mathcal{SH}(S))_{Q/} \cong Alg_{E_{\infty}}(Mod_Q)$ and that $R\cong \tau_{\geq 0} R$ is connective with respect to the $t$-structure defined in the previous part.
\end{hypothesis} 

Consider the Postnikov tower for $R$ with respect to the $t$-structure on $Q$-modules introduced before:

\[
    R \to \cdots \to R_{\leq n} \to R_{\leq n-1} \to \cdots \to R_{\leq 1} \to R_{\leq 0}
\]
As $R \cong \tau_{\geq 0}R$ is in the non-negative part of the $t$-structure, for negative indices $n$ we have $ R_{\leq n} \cong *$ is the point. As remarked after the proof of \ref{prop:prod.on.t.str}, this filtration lives in $E_{\infty}$ motivic ring spectra. 

Consider the derived push-outs: $Q_n:= Q\wedge_{R} R_{\leq n}$; we obtain a multiplicative filtration:
\[ 
    \lim_{\overleftarrow{n}} Q_{ n} \to \cdots Q_{n} \to Q_{n -1} \to \cdots \to Q_0 \to *
\]

Denote by 
\[
F_n = \Sigma^{-n,0}fib(R_{\leq n} \to R_{\leq n-1});
\]
observe that:
\[
\begin{aligned}
    fib(Q_n \to Q_{n-1}) &\cong fib(Q\wedge_{R} R_{\leq n} \to Q\wedge_{R} R_{\leq n-1}) \\
    &\cong Q\wedge_{R}fib(R_{\leq n} \to R_{\leq n-1}) \cong Q \wedge_R F_n[n].
\end{aligned}
\]

We study the homotopy spectral sequence arising from the  $Q_i$ \cite[Proposition 1.2.2.7]{Lurie2017}; the exact couple (see as a reference on the terminology \cite[Section 2.2]{McCleary2000}) is given by:
\begin{equation}\label{eqn:exact.couple}
    D^2_{s,(t,*)}=\pi_{s+t,*}(Q_{t}) \qquad 
    E^2_{s,(t,*)}=\pi_{s+t,*}(Q \wedge_{R} F_t[t])
\end{equation}
with $d^2$ differentials induced by:
\[
    Q \wedge_{R} F_t[t] \to Q_{t} \to Q \wedge_{R} F_{t+1}[t+2].
\]
We depict the situation in the commutative diagram \ref{fig:spectra.for.sequence}.

\begin{figure}
    \centering
    \begin{tikzcd}
     & \vdots \arrow[d] & \vdots \arrow[d] & \\
     Q \wedge_{R} F_{t+1}[t+1] \arrow[r, purple] 
     & Q_{ t +1} \arrow[d] \arrow[r] \arrow[ur, "id", purple] 
     & Q_{ t} \arrow[d] \arrow[r, blue] 
     & Q \wedge_{R} F_{t+1}[t+2] \\
     Q \wedge_{R} F_{t}[t] \arrow[r, blue] 
     & Q_{t} \arrow[d] \arrow[r] \arrow[ur, "id", blue] 
     & Q_{t-1} \arrow[d] \arrow[r, violet] 
     & Q \wedge_{R} F_{t}[t+1] \\
     Q \wedge_{R} F_{t-1}[t-1] \arrow[r,violet] 
     & Q_{ t-1} \arrow[d] \arrow[r] \arrow[ur, "id", violet] 
     & Q_{t-2} \arrow[d] \arrow[r, orange] 
     & Q \wedge_{R} F_{t-1}[t] \\     
     & \vdots \arrow[ur, "id", orange]& \vdots  & \\
    \end{tikzcd}
    \caption{The commutative diagram representing how the spectral sequence arises. Just follow the coloured paths to obtain the $d^2$ differentials.}
    \label{fig:spectra.for.sequence}
\end{figure}

Observe in particular that $Q_n \cong *$  is contractible for $n < 0$, due to $R$ being connective.

We can then apply \cite[Proposition 1.2.2.14]{Lurie2017} to get:
\begin{prop} \label{prop:ugly.spectral.sequence}
In the context of \ref{hyp:map.of.spectra}, there is a strongly convergent multiplicative spectral sequence:
\begin{equation} \label{eqn:ugly.spectral.sequence}
    E^2_{s,t,*}= \pi_{s+t,*}(Q \wedge_{R} F_t[t]) \Rightarrow \pi_{s+t,*}(\lim_{\overleftarrow{n}} Q_n)
\end{equation}

with differentials of the form
\[
    d^r: E^r_{s,t,*} \to E^r_{s-r,t+r-1,*}.
\]
\end{prop}

\begin{rmk}
    Observe that, as $R$ is connective with respect to our t-structure, the spectral sequence is non-zero for positive $t$; in other words, it lives in the upper half-plane. It is also worth noticing that the differentials do not alter the third index: we might interpret it as having a countable family of spectral sequences, indexed by the weight (although in this case, we would lose the ring structure of the pages).
\end{rmk}

We pass now to refining this result under stronger assumptions. We first look at the convergence term.

\begin{lemma} \label{lemma:limit.of.Qn}
    Assume, under hypothesis \ref{hyp:map.of.spectra}, that $\pi_{*,*}(Q)$ (and hence $\pi_{*,*}(R)$) is an algebra over some field $k$ (which is concentrated in degree $(0,0)$) and that $\pi_{i,*}(R) \cong \pi_{i,*}(Q) \cong 0$ for $i<0$.     
    Then there is an isomorphism of bi-graded rings:
    \[
        \pi_{*,*}(\lim_{\overleftarrow{n}} Q_n) \cong \pi_{*,*}(Q).
    \]
\end{lemma}

\begin{proof}
    We set up a Tor spectral sequence as in \cite[Proposition 7.7]{DugIsa2005}: it is a strongly convergent tri-graded spectral sequence
    \[
        E^2_{a,(b,c)} = Tor^{\pi_{*,*}R}_{a,(b,c)}(\pi_{*,*}Q, \pi_{*,*}R_{\leq n}) \Rightarrow \pi_{a+b,c}(Q \wedge_{R} R_{\leq n})
    \]
    For the rest of this proof, to improve readability, we use the following notation for motivic homotopy rings: $Q_{\star}$ will mean $\pi_{*,*} Q$, and similarly for the other spectra. 
    We analyse the $Tor$ modules appearing in the $E^2$ page, expressing them as the homology of the bar complex on $R_{\star}$ seen as a $k$-algebra (working over a field ensures that the $Bar$ complex provides a flat resolution):
    \begin{equation}\label{eqn:tor.spec.seq.Q.R}
        Tor^{R_{\star}}_{\bullet,(\star)}( Q_{\star}, R_{\leq n,\star}) =H_{\bullet}( Bar( Q_{\star},R_{\star},R_{\leq n,\star}))
    \end{equation}
    where:
    \begin{align*}
       Bar( Q_{\star},R_{\star},R_{\leq n,\star})=& Q_{\star}\otimes_k R_{\leq n,\star} \leftarrow  Q_{\star}\otimes_k R_{\star} \otimes_k R_{\leq n,\star} \\ 
       &\leftarrow  Q_{\star}\otimes_k R_{\star} \otimes_k R_{\star} \otimes_k R_{\leq n,\star} \leftarrow \ldots
    \end{align*}
    We wish to compare now this with the complex $Bar( Q_{\star},R_{\star},R_{\star})$ computing $Tor^{R_{\star}}_{\bullet,(\star)}( Q_{\star}, R)$; as $\pi_{i,*}(Q) \cong 0$ for $i<0$, we deduce as in \ref{eqn:homotopyidentities3}:
    \[
        \pi_{b,c}(R_{\leq n}) \cong \pi_{b,c}(R) \text{ for } b \leq n
    \]
    In particular, also $\pi_{b,c}(R_{\leq n})$ is connective in the first degree. Now recall that for all $m \geq 0$ we can write:
    \begin{align*}
        (Q_{\star}\otimes_k R_{\star}^{\otimes_k m} &\otimes_k R_{\leq n,\star})_{b,c}
        \\&\cong \bigoplus_{d,e \in \mathbb{Z}} (Q_{\star}\otimes_k R_{\star}^{\otimes_k m})_{b-d,c-e} \otimes_k R_{\leq n,(d,e)}
    \intertext{All the spectra involved are connective, we can bound the index $d$ to the terms with non-zero homotopy, so:}
    & \cong \bigoplus_{\substack{0 \leq d \leq b \\e \in \mathbb{Z}}} (Q_{\star}\otimes_k R_{\star}^{\otimes_k m})_{b-d,c-e} \otimes_k R_{\leq n,(d,e)}
    \intertext{If $b \leq n$:}
    & \cong \bigoplus_{\substack{0 \leq d \leq b \\e \in \mathbb{Z}}} (Q_{\star}\otimes_k R_{\star}^{\otimes_k m})_{b-d,c-e} \otimes_k R_{(d,e)}\\
    & \cong (Q_{\star}\otimes_k R_{\star}^{\otimes_k m} \otimes_k R_{\star})_{b,c}
    \end{align*}

    As then the two bar complexes:
    
    \begin{adjustbox}{max width=\textwidth, center}
        \begin{tikzcd}
        Q_{\star}\otimes_k R_{\leq n,\star} 
        &Q_{\star}\otimes_k R_{\star} \otimes_k R_{\leq n,\star} 
            \arrow{l}
        &  Q_{\star}\otimes_k R_{\star} \otimes_k  R_{\star} \otimes_k R_{\leq n,\star} 
            \arrow{l}
        & \ldots
            \arrow{l}
        \\ Q_{\star}\otimes_k R_{\star}
            \arrow{u}
        &  Q_{\star}\otimes_k R_{\star} \otimes_k R_{\star}
            \arrow{u}
            \arrow{l}
        &  Q_{\star}\otimes_k R_{\star} \otimes_k  R_{\star} \otimes_k R_{\star}
            \arrow{u}
            \arrow{l}
        & \ldots
            \arrow{l}
    \end{tikzcd}
    \end{adjustbox}
    
    are isomorphic up to degree $n$, also their homologies will be in the same range:
    \[
        Tor^{R_{\star}}_{a,(b,c)}(Q_{\star}, R_{\leq n, \star}) \cong Tor^{R_{\star}}_{a,(b,c)}(Q_{\star}, R_{\star})
    \]
    Now, $Tor^{R_{\star}}_{a,(b,c)}(Q_{\star}, R_{\star})\cong 0$ for all $a>0$, as $R_{\star}$ is flat as a $R_{\star}$-module, while:
    \[
        Tor^{R_{\star}}_{0,(b,c)}(Q_{\star}, R_{\star})\cong Q_{\star} \otimes_{R_{\star}} R_{\star} \cong Q_{\star}.
    \]
    We see this depicted in figure \ref{fig:visual.spec.seq}.
    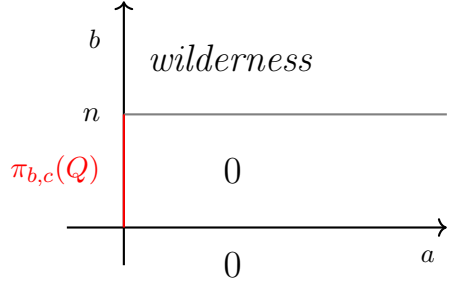
\begin{figure}[htb]
        \centering

\begin{tikzpicture}[scale=0.5,line width=1pt]

\draw[black, line width=0.8pt,->](-1.5,0)--(8.5,0);
\draw[black, line width=0.8pt,->] (0,-1)--(0,6);

\node[label={180:{\footnotesize $b$}}] at (0,5){};
\node[label={270:{\footnotesize $a$}}] at (8,0){};

\draw[gray, line width=0.8pt] (0,3)--(8.5,3);

\draw[red, line width=0.8pt] (0,0)--(0,3);

\node[label={[red] 180:{\small $\pi_{b,c}(Q)$}}] at (0,1.5){};

\node[label={180:{\small $n$}}] at (0,3){};

\node[label={0:{\large 0}}] at (2,1.5){};
\node[label={0:{\large 0}}] at (2,-1){};
\node[label={0:{\large \textit{wilderness}}}] at (0,4.5){};
\end{tikzpicture}

\caption{A visual representation of the $Tor$ spectral sequence \ref{eqn:tor.spec.seq.Q.R}.}
        \label{fig:visual.spec.seq}
    \end{figure}
    
    But then the spectral sequence is trivial for $b \leq n$, and we can identify:
    \[
        \pi_{b,c} (Q \wedge_{R} R_{\leq n}) \cong \pi_{b,c} Q
    \]
    for $b \leq n$ and all $c$. Taking the limit for $n \rightarrow \infty$ produces the desired result.
     
\end{proof}

Next, we aim at having a nicer $E^2$ term for \ref{eqn:ugly.spectral.sequence}.

\begin{lemma} \label{lemma:smash.becomes.tensor}
    Under hypothesis \ref{hyp:map.of.spectra}, assume that the shifted fibres $F_n$ are flat over $Q$.
    
    Then we have isomorphisms:
    \[
        \pi_{s+t,*}(Q \wedge_{R} F_t[t]) \cong \pi_{s,*}P\otimes_{\pi_{0,*}Q} \pi_{0,*}F_t
    \] 
    with $P=Q\wedge_R Q$. 
\end{lemma}

\begin{proof}
    Since we work in $Q$-modules, we can write:
    \[
        Q \wedge_{R} F_n[n] \simeq
        Q \wedge_{R} Q \wedge_{Q} F_n[n]      \simeq
        P\wedge_{Q} F_n[n].
    \]
    By the flatness assumption, by \cite[Proposition 7.2.2.13]{Lurie2017}:
     \[
        \pi_{s+t,*}(P\wedge_{Q} F_t[t])\simeq\pi_{s,*}(P\wedge_{Q} F_t)\simeq \pi_{s,*} P \otimes_{\pi_{0,*}Q} \pi_{0,*}F.
    \]
\end{proof}

\begin{rmk}
    Observe that flatness, as formulated in \cite[Definition 7.2.2.10]{Lurie2017}, is just a condition on the homotopy groups of the objects involved; it is fairly easy to check it when working with explicit objects, as in our application.
\end{rmk}
\clearpage

\section{\texorpdfstring{Motivic Hochschild homology, $\tau$ inverted.}{Motivic Hochschild homology, tau inverted.}} \label{sec:MHH.tau-1}

Consider now $Q=M\mathbb{Z}/p$, the modulo $p$ motivic cohomology spectrum, and $R=M\mathbb{Z}/p \wedge M\mathbb{Z}/p=\mathcal{A}(p)$, the dual modulo $p$ motivic Steenrod algebra spectrum in the stable homotopy category $\mathcal{SH}(S)$. Observe that these satisfy all the hypothesis of the previous section; in particular connectivity of $\mathcal{A}(p)$ is granted by the $M\mathbb{Z}/p$-module description appearing in \cite[Theorem 1.1]{HKO2013}. If $S=Spec(F)$ is the spectrum of an algebraically closed field of characteristic different than $p$, on the homotopy groups level we have:
\begin{equation}\label{eqn:homotopy.of.MZp.Ap}
\begin{gathered}
    \pi_{*,*}(M\mathbb{Z}/p)  \simeq \mathbb{F}_p[\tau] \text{ with } \left |\tau \right |=(0,-1)\\
    \pi_{*,*}(\mathcal{A}(p))  \simeq 
    \begin{cases}
        \mathbb{F}_2[\tau, \xi_i, \tau_i]_{i \geq 0}/(\tau_i^2-\tau \xi_{i+1}) & \text{ for } p=2\\
        \mathbb{F}_p[\tau, \xi_i, \tau_i]_{i \geq 0}/(\tau_i^2) & \text{ for } p=\text{odd}.
    \end{cases}
\end{gathered}
\end{equation} 
The degrees of the generators are:
\[\arraycolsep=15pt
\begin{array}{lcr}
     |\tau|=(0,-1) & |\tau_i|=(2p^i-1,p^i-1) & |\xi_i|=(2p^i-2, p^i-1).
\end{array}
\]

Notice that the homotopy groups of $M \mathbb{Z}/p$ are concentrated on a vertical line and that $\pi_{\star}(\mathcal{A}(p))$ is a free (in particular: flat) $\pi_{\star}(M \mathbb{Z}/p)$-module.
We define the motivic Hochschild homology of $M\mathbb{Z}/p$ as the derived product:
\[
    MHH(M\mathbb{Z}/p)= M\mathbb{Z}/p \wedge_{M\mathbb{Z}/p \wedge M\mathbb{Z}/p} M\mathbb{Z}/p.
\]
in analogy to what is done for topological Hochschild homology. 

Our aim for this section is to compute the homotopy structure of a closely related object, namely, the one where we add a homotopy inverse to the element $\tau$:
\[
    MHH(M\mathbb{Z}/p)[\tau^{-1}]:= M\mathbb{Z}/p[\tau^{-1}] \wedge_{(M\mathbb{Z}/p \wedge M\mathbb{Z}/p)[\tau^{-1}]} M\mathbb{Z}/p[\tau^{-1}].
\]
Our main reference for inverting homotopy classes at the level of spectra is \cite[Appendix B]{BEO2020}. One may also have a look at \cite{Arthan1983}, where the problem was approached before the $\infty$-categorical machinery. 
We begin by reviewing such constructions.

Given the element $\tau \in \pi_{0,-1}(M\mathbb{Z}/p)=[S^{0,-1},M\mathbb{Z}/p]$, we pick the map (in fact a map, since we may choose it up to homotopy equivalence):
\[
    S^{0,-1}\wedge M\mathbb{Z}/p \xrightarrow{\tau \wedge id_M} M\mathbb{Z}/p\wedge M\mathbb{Z}/p \xrightarrow{\mu} M\mathbb{Z}/p,
\]
$\mu$ being the multiplication. Let $\epsilon: M\mathbb{Z}/p \to  S^{0,1}\wedge M\mathbb{Z}/p $ be an adjoint map to it, which we can identify with $S^{0,1} \wedge(\mu \circ (\tau\wedge id_M))$.

\begin{rmk}
    Unluckily, there is a slightly different notation in \cite{BEO2020}, in particular in \cite[Lemma B.1]{BEO2020}, which is the theoretical result behind this localisation. In particular, they call $\tau$ the map we denote by $\epsilon$ (but still denote the localization of a ring $E$ at the homotopy element $\tau$ as $E[\tau^{-1}]$).
\end{rmk}

We then define the spectrum $M\mathbb{Z}/p[\tau^{-1}]$ as the mapping telescope:

\[\begin{adjustbox}{max width=\textwidth, center}
        $\displaystyle M\mathbb{Z}/p[\tau^{-1}]:=\colim\left(M\mathbb{Z}/p \xrightarrow{\epsilon} S^{0,1}\wedge M\mathbb{Z}/p \xrightarrow{id_{S^{0,1}}\wedge\epsilon} S^{0,2}\wedge M\mathbb{Z}/p \xrightarrow{id_{S^{0,2}}\wedge\epsilon}\cdots\right)$
\end{adjustbox}\]

\begin{rmk}
Since this is a colimit in the $\infty$-categorical sense, one should describe a diagram $N(\mathbb{N}) \to \mathcal{SH}(F)$; however, as observed in the proof of \cite[Lemma B.1]{BEO2020}, the above description implies the correct one.
\end{rmk}

Observe that the element $\tau$ lives in the dual motivic Steenrod algebra as well (where we name it $\tau_{\mathcal{A}}$ for the moment), via the ring map $M\mathbb{Z}/p \to \mathcal{A}(p)$:
\[
 S^{0,-1} \xrightarrow{\tau} M \mathbb{Z}/p \cong M \mathbb{Z}/p \wedge \mathbb{S} \xrightarrow{id_M \wedge i} M \mathbb{Z}/p \wedge M \mathbb{Z}/p = \mathcal{A}(p)
\]
where $i: \mathbb{S} \to M \mathbb{Z}/p$ is the canonical map from the zero object.

We apply then the same localization procedure to the adjoint map $\epsilon_{\mathcal{A}}: \mathcal{A}(p) \to S^{0,1} \wedge \mathcal{A}(p)$ to get:
\[
    \mathcal{A}(p)[\tau_{\mathcal{A}}^{-1}]=(M\mathbb{Z}/p \wedge M\mathbb{Z}/p)[\tau_{\mathcal{A}}^{-1}].
\]

The two localisations are compatible in the following sense. 
First, one has a commutative diagram:

\[\begin{adjustbox}{max width=\textwidth, center}
    \begin{tikzcd}
        S^{0,-1}\wedge \mathcal{A}(p) 
            \arrow[r, "\tau \wedge Id_{\mathcal{A}}"] 
            \arrow[d,"Id_S \wedge \mu_M"] 
            \arrow[rr, bend left, "\tau_{\mathcal{A}} \wedge id_{\mathcal{A}}"]
        & M\mathbb{Z}/p \wedge \mathcal{A}(p) 
            \arrow[r, "i \wedge id_{\mathcal{A}}"] 
            \arrow[d,"id_M \wedge \mu_M"]
        & \mathcal{A}(p) \wedge \mathcal{A}(p) 
            \arrow[r, "\mu_{\mathcal{A}}"]
        & \mathcal{A}(p) 
            \arrow[d,"\mu_M"] 
        \\ S^{0,-1}\wedge M\mathbb{Z}/p 
            \arrow[r, "\tau \wedge Id_{M}"]
        & M\mathbb{Z}/p\wedge M\mathbb{Z}/p 
            \arrow[rr, "\mu_{M}"] 
        && M\mathbb{Z}/p  .
    \end{tikzcd}
\end{adjustbox}\]

The half moon on top commutes by the definition of $\tau_{\mathcal{A}}$, the left square commutes because the map acts independently on the two factors, and the right square commutes because of the compatibility between the ring structures of $\mathcal{A}(p)$ and $M\mathbb{Z}/p$.
Then we have a commutative square:
\[
\begin{tikzcd}
    \mathcal{A} 
        \arrow[r, "\epsilon_{\mathcal{A}}"]
        \arrow[d, "\mu"]
    & S^{0,1} \wedge \mathcal{A}  
        \arrow[d, "id \wedge \mu"] 
    \\ M \mathbb{Z}/p
        \arrow[r, "\epsilon_{M}"]
    & S^{0,1} \wedge M \mathbb{Z}/p
\end{tikzcd}
\]
We can hence suppress the distinction between $\tau_{\mathcal{A}}$ and $\tau$. Applying the maps $\epsilon_M$ and $\epsilon_{\mathcal{A}}$ iteratively, we get a commutative diagram of telescopes of ring spectra:
\begin{center}
    \begin{tikzcd}
        \mathcal{A}(p) 
        \arrow[r, "\epsilon_{\mathcal{A}}"]
        \arrow[d, "\mu_M"] 
        &  S^{0,1}\wedge \mathcal{A}(p) \arrow[r, "Id_{S} \wedge \epsilon_{\mathcal{A}}"] \arrow[d, "Id_{S}\wedge \mu_M"] 
        &  S^{0,2}\wedge \mathcal{A}(p) \arrow[r, "Id_{S} \wedge \epsilon_{\mathcal{A}}"] \arrow[d, "Id_{S}\wedge \mu_M"] 
        &\cdots
        \\ M\mathbb{Z}/p \arrow[r, "\epsilon_{M}"] 
        &  S^{0,1}\wedge M\mathbb{Z}/p \arrow[r, "Id_{S} \wedge \epsilon_{M}"]
        &  S^{0,2}\wedge M\mathbb{Z}/p \arrow[r, "Id_{S} \wedge \epsilon_{M}"]
        &\cdots
    \end{tikzcd}
\end{center}
where by $Id_S$ we mean the identity on whatever sphere we have. But then we have a map of rings between the colimits:
\[
    \mathcal{A}(p)[\tau^{-1}] \to M\mathbb{Z}/p[\tau^{-1}].
\]

Using instead of the multiplication $\mu_M: \mathcal{A}(p) \to M \mathbb{Z}/p$ the inclusion $M \mathbb{Z}/p \to \mathcal{A}(p)$ we get also a map in the other way:
\[
    M\mathbb{Z}/p[\tau^{-1}] \to \mathcal{A}(p)[\tau^{-1}].
\]
This procedure, in addition, preserves the module decomposition of $\mathcal{A}(p)$ appearing in \cite[Theorem 1.1]{HKO2013}, in the sense that $\mathcal{A}(p)[\tau^{-1}]$ can be decomposed as the union of shifted copies of $M\mathbb{Z}/p[\tau^{-1}]$, with the same indices as those appearing in the already mentioned theorem. In particular, $\mathcal{A}(p)[\tau^{-1}]$ is connective in our $t$-structure. This can also be proven with \ref{eqn:homotopyidentities1} and \ref{eqn:homotopyidentities3} ($M\mathbb{Z}/p[\tau^{-1}]$ still has homotopy in a single vertical line) by simple inspection of the homotopy groups, as the spectra involved are cellular.
    
Hence hypothesis \ref{hyp:map.of.spectra} is verified, so, by \ref{prop:ugly.spectral.sequence}, we have:
\begin{prop}\label{prop:ugly.MHH.tau-1.sp.seq}
    There is a strongly convergent, upper-half plane multiplicative spectral sequence:
\[
    E^2_{s,t,*}= \pi_{s+t,*}(M\mathbb{Z}/p[\tau^{-1}] \wedge_{\mathcal{A}(p)[\tau^{-1}]} F_t[t]) \Rightarrow \pi_{s+t,*}(\lim_{\overleftarrow{n}} M\mathbb{Z}/p[\tau^{-1}]_n)
\]

with differentials of the form
\[
    d^r: E^r_{s,t,*} \to E^r_{s-r,t+r-1,*}.
\]
where 
\[
    F_t[t]=fib((\mathcal{A}(p)[\tau^{-1}])_{\leq t} \to (\mathcal{A}(p)[\tau^{-1}])_{\leq t-1})
\]
and 
\[
    M\mathbb{Z}/p[\tau^{-1}]_n = M\mathbb{Z}/p[\tau^{-1}] \wedge_{\mathcal{A}(p)[\tau^{-1}]} (\mathcal{A}(p)[\tau^{-1}])_{\leq n}.
\]
\end{prop}

Now, inverting $\tau$ for our spectra means that their homotopy rings become (confront with \ref{eqn:homotopy.of.MZp.Ap}):
\begin{equation}\label{eqn:homotopy.MZp.A.tau-1}
    \begin{gathered}
        \pi_{*,*}(M\mathbb{Z}/p)[\tau^{-1}]  \simeq \mathbb{F}_p[\tau^{\pm 1}]\\
        \pi_{*,*}(\mathcal{A}(p)) [\tau^{-1}]  \simeq 
        \begin{cases}
            \mathbb{F}_2[\tau^{\pm 1}, \tau_i]_{i \geq 0} & \text{ for } p=2\\
            \mathbb{F}_p[\tau^{\pm 1}, \xi_i, \tau_i]_{i \geq 0}/(\tau_i^2) & \text{ for } p=\text{odd}.
        \end{cases}
    \end{gathered}
\end{equation}

As before, we want to
get a nicer-looking spectral sequence to make computations easier. From \ref{eqn:homotopy.MZp.A.tau-1}, we can see that $\pi_{\star}(M\mathbb{Z}/p[\tau^{-1}])$ is concentrated in the vertical line of first degree equal to 0; in particular, it vanishes in negative degrees. A quick look at the homotopy groups shows that both $\pi_{\star}(M\mathbb{Z}/p[\tau^{-1}])$ and $\pi_{\star}(\mathcal{A}(p)[\tau^{-1}])$ are algebras over the field $\mathbb{F}_p$, so lemma \ref{lemma:limit.of.Qn} applies:  $\pi_{s+t,*}(\lim\limits_{\overleftarrow{n}} M\mathbb{Z}/p[\tau^{-1}]_n) \cong \pi_{s+t,*}(M\mathbb{Z}/p[\tau^{-1}])$.

Observing again the ring structure of $\pi_{*,*}(\mathcal{A}(p)) [\tau^{-1}]$ and recalling that, as the homotopy of $M\mathbb{Z}/p)[\tau^{-1}$ is concentrated in our setting in the zeroth column, truncating with respect to our $t$-structure is the same as truncating with respect to the first degree (\ref{eqn:homotopy.of.fibre1} and \ref{eqn:homotopy.of.fibre2}), one notices that the shifted fibres $F_t$ are flat over $M\mathbb{Z}/p[\tau^{-1}]$ (in particular they are again free modules). So lemma \ref{lemma:smash.becomes.tensor} applies:
\[
\begin{adjustbox}{max width=\textwidth, center}
    $\displaystyle
    \pi_{s+t,*}(M\mathbb{Z}/p[\tau^{-1}] \wedge_{\mathcal{A}(p)[\tau^{-1}]} F_t[t]) \cong \pi_{s,*}MHH(M\mathbb{Z}/p)[\tau^{-1}]\otimes_{\pi_{0,*}\mathcal{A}(p)[\tau^{-1}]}M\mathbb{Z}/p[\tau^{-1}].$
\end{adjustbox}
\]

Finally, notice that $M\mathbb{Z}/p[\tau^{-1}]$ is the unit in commutative algebras over $M\mathbb{Z}/p[\tau^{-1}]$, so the pushout square:
\[
\begin{tikzcd}
    \mathcal{A}(p)[\tau^{-1}]
        \arrow{d}
        \arrow{r}
    &M\mathbb{Z}/p[\tau^{-1}]
        \arrow{d}
    \\M\mathbb{Z}/p[\tau^{-1}]
        \arrow{r}
    &MHH(M\mathbb{Z}/p)[\tau^{-1}]
\end{tikzcd}
\]
provides a (co) fibre sequence:
\[
    M\mathbb{Z}/p[\tau^{-1}] \to MHH(M\mathbb{Z}/p)[\tau^{-1}] \to \Sigma^{1,0}\mathcal{A}(p)[\tau^{-1}].
\]
As $M\mathbb{Z}/p[\tau^{-1}]$ and $\Sigma^{1,0}\mathcal{A}(p)[\tau^{-1}]$ belong to the non-negative part of the $t$-structure, which is closed under extensions, also $MHH(M\mathbb{Z}/p)[\tau^{-1}]$ will; this identifies, in our context, with having homotopy groups concentrated in positive degrees.
Putting it all together:

\begin{prop}\label{prop:spectral.seq.tau.-1}
Let $F$ be an algebraically closed field, and let $p$ be a prime number, $p\neq char(F)$. 
    \begin{itemize}
        \item Let $M\mathbb{Z}/p[\tau^{-1}]$ be the colimit of:
        \[
            M\mathbb{Z}/p \to S^{0,1} \wedge M\mathbb{Z}/p \to S^{0,2} \wedge M\mathbb{Z}/p \to \cdots
        \]
        where the horizontal maps are induced by the homotopy element $\tau: S^{0,-1} \to M\mathbb{Z}/p$.
        \item Let $\mathcal{A}(p)[\tau^{-1}]$ be the colimit of:
        \[
            \mathcal{A}(p) \to S^{0,1} \wedge \mathcal{A}(p) \to S^{0,2} \wedge \mathcal{A}(p) \to \cdots
        \]
        where the horizontal maps are induced by the homotopy element $\tau: S^{0,-1} \to \mathcal{A}(p)$.
        \item Let
        \[
            MHH(M\mathbb{Z}/p)[\tau^{-1}]=M\mathbb{Z}/p[\tau^{-1}]\wedge_{\mathcal{A}(p)[\tau^{-1}]}M\mathbb{Z}/p[\tau^{-1}]
        \]
        be their pushout.
    \end{itemize} 
    Then we have a first-quadrant multiplicative spectral sequence:
\[
\begin{adjustbox}{max width=\textwidth, center}
$\displaystyle
   E^2_{s,t,*}=\pi_{s,*}(MHH(M\mathbb{Z}/p)[\tau^{-1}]) \otimes_{\pi_{0,*}(M\mathbb{Z}/p[\tau^{-1}])} \pi_{t,*}(\mathcal{A}(p)[\tau^{-1}]) \Rightarrow \pi_{s+t,*}(M\mathbb{Z}/p[\tau^{-1}]).
$    
\end{adjustbox}
\]
with differentials:
\[
    d^k=d^k_{s,t,*}:E^k_{s,t,*} \to E^k_{s-k,t+k-1,*}.
\]
\end{prop}

Observe in particular that the convergence term $\pi_{s+t,*}(M\mathbb{Z}/p[\tau^{-1}])$ is non-zero only for $s+t=0$. As the spectral sequence is first-quadrant, in the $E^{\infty}$ page we must have:
\[
E^{\infty}_{s,t,*} \cong 
\begin{cases}
    \pi_{0,*}(M\mathbb{Z}/p[\tau^{-1}]) &\text{ for } s=t=0\\
    0 &\text{ elsewhere.}
\end{cases}
\]

\clearpage

\subsection{The computation} \label{subs:computation}

Before proceeding with the computation of $\pi_{*,*}MHH(M\mathbb{Z}/p)[\tau^{-1}]$, we first need to introduce a further instrument to help us to deal, in particular, with the multiplicative structure of the homotopy rings involved. We define the suspension map:
\begin{equation}
  \begin{array}{ll}
   \sigma: &\Sigma^{1,0}(M\mathbb{F}_p \wedge M\mathbb{F}_p) \to S^1_+\wedge M\mathbb{F}_p \wedge M\mathbb{F}_p \\
    &\to M\mathbb{F}_p \wedge (S^1 \otimes M\mathbb{F}_p) \cong M\mathbb{F}_p \wedge MHH(\mathbb{F}_p) \\
    &\xrightarrow{mult.} MHH(\mathbb{F}_p) 
  \end{array}  
\end{equation}

Seen as a pushforward on homotopy groups, it gives a map:
\begin{equation}
   \sigma_*: \mathcal{A}(p)_{s-1,t}\to  MHH_{s,t}(\mathbb{F}_p) 
\end{equation}
whose action on certain classes can be explicited (\cite{DHOO2022}, Lemma 2.3):
\begin{prop} \label{prop:mot.relations}
Let $F$ be an algebraically closed field of characteristic different that $p$ and let $MHH(M\mathbb{Z}/p) \in \mathcal{SH}(F)$ be the motivic Hochschild homology spectrum of $M\mathbb{Z}/p$; in  $MHH_{\star}(\mathbb{F}_p)$ we have the relations:
\begin{equation*}
    \tau^{p-1} \sigma_*\tau_{i+1}=(\sigma_*\tau_i)^p, \qquad \tau^{p-1} \sigma_*\xi_{i+1}=0
\end{equation*}
for all $i \geq 0$.
\end{prop}

By repeating the above construction with the $\tau$-inverted spectra, the above result descends to our context of interest:
\begin{cor}\label{cor:power.operations}
    There is a map of spectra 
    \[
        \sigma: \Sigma^{1,0}\mathcal{A}(p)[\tau^{-1}] \to MHH(M\mathbb{Z}/p)[\tau^{-1}]
    \]
    such that on the homotopy rings it induces relations:
    \begin{equation}
    \tau^{p-1} \sigma_*\tau_{i+1}=(\sigma_*\tau_i)^p, \qquad \sigma_*\xi_{i+1}=0
\end{equation}
for all $i \geq 0$.
\end{cor}

The map $\sigma$ is of particular interest to us because it is deeply linked to our spectral sequence. In particular, since the spectral sequence $\ref{prop:spectral.seq.tau.-1}$ is first quadrant and the convergence term is concentrated at the origin, for each $q \geq 2$ there is an isomorphism:
\[
    E^q_{q,0,*} \xrightarrow{d^q} E^q_{0,q-1,*}.
\]
This particular kind of differential is called transgressive. We show that these transgressive differentials provide a partial inverse to $\sigma_*$.

\begin{prop}\label{prop:magic.square}
The following diagram commutes:
\begin{equation}\label{eqn:magic.square}
\begin{tikzcd}
E^q_{q,0,*} \arrow[r,"d^q","\simeq"'] \arrow[d, hook] & E^q_{0,q-1,*}\\
\pi_{q,*} MHH(M\mathbb{Z}/p)[\tau^{-1}] & \pi_{q-1,*} \mathcal{A}(p)[\tau^{-1}] \arrow[l,"\sigma_*"] \arrow[u, two heads]
\end{tikzcd}
\end{equation}
\end{prop}

\begin{proof}
    For this proof, we use the substitute symbols $\mathcal{A}=\mathcal{A}(p)[\tau^{-1}]$ and $\mathcal{M}=M\mathbb{Z}/p[\tau^{-1}]$ to improve readability.

    Given the a map $\mathcal{A} \to \mathcal{M}$ of connective commutative motivic ring spectra over $\mathcal{M}$ we form the pushouts:
    \[
    \begin{tikzcd}
    \mathcal{A} \arrow[r] \arrow[d] & \mathcal{M} \arrow[d]\\
    \mathcal{M} \arrow[r] & \mathcal{M}\oplus_{\mathcal{A}} \mathcal{M}
    \end{tikzcd}
    \qquad
    \begin{tikzcd}
    \mathcal{A} \arrow[r] \arrow[d] & \mathcal{M} \arrow[d]\\
    \mathcal{M} \arrow[r] & \mathcal{M}\otimes_{\mathcal{A}} \mathcal{M}
    \end{tikzcd}
    \]
        respectively in in $Mod_{\mathcal{A}}$ and in $ CAlg(Mod_{\mathcal{A}})\cong CAlg(SH(S))_{\mathcal{A}/}$ respectively. Via the forgetful functor, we can see $\mathcal{M}\otimes_{\mathcal{A}} \mathcal{M}$ as an element of $Mod_{\mathcal{A}}$, hence we deduce natural maps:
    \[
    \begin{tikzcd}
    \mathcal{A} \arrow[r] \arrow[d] & \mathcal{M} \arrow[d] \arrow[rdd]&\\
    \mathcal{M} \arrow[r] \arrow[rrd] & \mathcal{M}\oplus_{\mathcal{A}} \mathcal{M} \arrow[rd]&\\
    &&  \mathcal{M} \otimes_{\mathcal{A}} \mathcal{M}
    \end{tikzcd}
    \]
    
    Let $\tilde{\mathcal{A}}=fib(\mathcal{A} \to \mathcal{M})$. Given the shape of $\mathcal{A}$ and properties of our $t$-structure, from the point of view of homotopy groups $\tilde{\mathcal{A}}$ is the same as $\mathcal{A}$, but in degree 0, where it is trivial. 
    Consider again the pushout:
    \[
    \begin{tikzcd}
    \mathcal{A} \arrow[r] \arrow[d] & \mathcal{M} \arrow[d]\\
    \mathcal{M} \arrow[r] & \mathcal{M} \oplus_{\mathcal{A}} \mathcal{M}
    \end{tikzcd}
    \]
    As cofibre sequences can be expressed as certain pushouts, the second column of the above pushout diagram will present the same cofibre as the first:
    \[
    \begin{tikzcd}
    \mathcal{A} \arrow[r] \arrow[d] & \mathcal{M} \arrow[d]\\
    \mathcal{M} \arrow[r] \arrow[d] & \mathcal{M} \oplus_{\mathcal{A}} \mathcal{M}\arrow[d] \\
    \Sigma^{1,0} \tilde{\mathcal{A}} & \Sigma^{1,0} \tilde{\mathcal{A}}
    \end{tikzcd}
    \]
    
    Given the structure of $\mathcal{A}$ (hence of $\tilde{\mathcal{A}}$) as an $\mathcal{M}$-module, we can conclude an equivalence $\mathcal{M} \oplus_{\mathcal{A}} \mathcal{M} \cong \mathcal{M} \oplus \Sigma^{1,0} \tilde{\mathcal{A}}$.
    
    Truncate now $\mathcal{A}$ according to the $t$-structure:
    \[
    \mathcal{A} \rightarrow \ldots \rightarrow \mathcal{A}_{\leq 2} \rightarrow \mathcal{A}_{\leq 1} \rightarrow \mathcal{A}_{\leq 0} \cong \mathcal{M}
    \]
    Define $C_i := \mathcal{A}_{\leq i} \oplus_{\mathcal{A}} \mathcal{M}$, $D_i:= \mathcal{A}_{\leq i} \otimes_{\mathcal{A}} \mathcal{M}$. Then we have a commutative diagram:
    \[
    \begin{adjustbox}{max width=\textwidth, center}
    \begin{tikzcd}
    \mathcal{A} \arrow[r] \arrow[d] & \mathcal{M} \arrow[r, "id"] \arrow[d] & \mathcal{M} \arrow[d]\\
    \vdots \arrow[d]&\vdots \arrow[d]& \vdots\arrow[d]\\
    \mathcal{A}_{\leq 2} \arrow[r] \arrow[d] & C_2 \arrow[r] \arrow[d] & D_2 \arrow[d]\\
    \mathcal{A}_{\leq 1} \arrow[r] \arrow[d] & C_1 \arrow[r] \arrow[d] & D_1 \arrow[d]\\
    \mathcal{A}_{\leq 0} \cong \mathcal{M} \arrow[r]  & C_0 \cong \mathcal{M} \oplus \Sigma^{1,0}\tilde{\mathcal{A}} \arrow[r]  & D_0\cong \mathcal{M} \otimes_{\mathcal{A}} \mathcal{M}
    \end{tikzcd}
    \end{adjustbox}
    \]
    Each of these towers gives rise to a spectral sequence. In particular, the $D_j$ give rise to the spectral sequence we introduced before:
    \[\begin{adjustbox}{max width=\textwidth, center}
    \parbox{\linewidth}{
    \begin{align*}
        E^2_{s,t,*}=& \pi_{s+t,*}(fib(D_t \to D_{t-1})) \\
        &\cong \pi_{s,*}(MHH(\mathcal{M})) \otimes_{\pi_{0,*}(\mathcal{M})} \pi_{t,*}(\mathcal{A}) \Rightarrow \pi_{s+t,*}(\mathcal{M}).
    \end{align*}}
    \end{adjustbox}\]
    
    The $C_j$ give rise instead to the following spectral sequence:
    \[\begin{adjustbox}{max width=\textwidth, center}
    \parbox{\linewidth}{\begin{align*}
        E'^2_{s,t,*}=& \pi_{s+t,*}(fib(C_t \to C_{t-1})) \\
        &=\pi_{s+t,*}(fib\left((\mathcal{A}_{\leq t}\oplus_{\mathcal{A}} \mathcal{M}) \to (\mathcal{A}_{\leq t-1} \oplus_{\mathcal{A}} \mathcal{M})\right))\\
        &\cong 
        \begin{cases}
            \pi_{s,*}(\mathcal{M}\oplus \Sigma^{1,0} \tilde{\mathcal{A}}) & t=0\\
            \pi_{s,*}(F_t) & t>0
        \end{cases}  
    \end{align*}}
    \end{adjustbox}\]
    where $\Sigma^{t,0}F_t=fib(\mathcal{A}_{\leq t} \to \mathcal{A}_{\leq t-1})$; recall that under our assumptions 
    \[\pi_{s,*}(F_t) \cong 
    \begin{cases}
        \pi_{s,*} \mathcal{A} &\text{ for } s=t\\
        0 & \text{ otherwise.}
    \end{cases}\]
    Notice that also this spectral sequence converges to the ring: $E'^2_{s,t,*} \rightarrow \pi_{s+t,*}(\mathcal{M})$.
    
    The sparsity of non-zero modules in the $E^2$-page of this second spectral sequence forces its behaviour to be very easy. Notice that it is first quadrant, and we have non-zero terms only along the axes. Since the convergence term is concentrated in the origin, we must have only one non-trivial differential $d^i:E'^i_{i,0,*} \to E'^i_{0,i-1,*}$ for each $i \geq 2$, and it must be an isomorphism. 
    
    Now, the map between the filtrations $C_i$ and $D_i$ produces a map between the spectral sequences. On the $E^2$ page, this looks like:
    \begin{figure}[H]
    \centering
    \adjustbox{max width= \textwidth, center}{\begin{tikzpicture}
    \draw[->, gray] (0,0)--(0,4.5);
    \draw[->, gray] (0,0)--(5.5,0);
    
    \draw[->, gray] (8,0)--(8,4.5);
    \draw[->, gray] (8,0)--(13.5,0);
    
    \draw[step=1 ,lightgray,very thin] (0,0) grid (5.5,4.5);
    \draw[step=1 ,lightgray,very thin] (8,0) grid (13.5,4.5);
    
    \foreach \y in {0,...,4}
    {
    \filldraw (0,\y) circle (1pt);
    }
    \foreach \x in {2,...,5}
    {
    \filldraw (\x,0) circle (1pt);
    }
    
    \foreach \x in {0,2,3,4,5}
    \foreach \y in {0,...,4}
    {
    \filldraw (\x+8,\y) circle (1pt);
    }
    
    \node foreach \y in {0,...,4} at (0, \y) [left] {\small $\pi_{\y,*}\mathcal{A}$};
    \node foreach \y in {0,...,4} at (8, \y) [left] {\small $\pi_{\y,*}\mathcal{A}$};
    \node foreach \x in {1,...,4} at (\x+1, 0) [below, rotate=90, anchor=east] {\small $\pi_{\x,*}\mathcal{A}$};
    \node foreach \x in {1,...,4} at (\x+9, 0) [below, rotate=90, anchor=east] {\small $\pi_{\x,*}MHH(\mathcal{M})$};

    \draw[->, line width=0.75pt, blue] (2,0)--(0,1) node[blue, below, midway]{\small $id$};

    \draw[->, line width=0.75pt, blue] (10,0)--(8,1) node[blue, below, midway]{\small $d^2$};

    \path (0,0) edge[->, line width=0.75pt, red, bend right=22] (8,0);
    \path (2,0) edge[->, line width=0.75pt, red, bend right=22] (10,0);
    \path (3,0) edge[->, line width=0.75pt, red, bend right=22] node[red, below, midway]{\small $\sigma_*$} (11,0);
    \path (3,0) edge[->, line width=0.75pt, red, bend right=22] (11,0);
    \path (4,0) edge[->, line width=0.75pt, red, bend right=22] (12,0);
    \path (5,0) edge[->, line width=0.75pt, red, bend right=22] (13,0);

    \path (0,1) edge[->, line width=0.75pt, teal, bend left=22, looseness=0.8] (8,1);
    \path (0,2) edge[->, line width=0.75pt, teal, bend left=22, looseness=0.8] (8,2);
    \path (0,3) edge[->, line width=0.75pt, teal, bend left=22, looseness=0.8] (8,3);
    \path (0,4) edge[->, line width=0.75pt, teal, bend left=22, looseness=0.8] node[teal, above , midway]{\small $id$}(8,4);
\end{tikzpicture}}
    \caption{The map of spectral sequences at the $E^2$ page}
\end{figure}
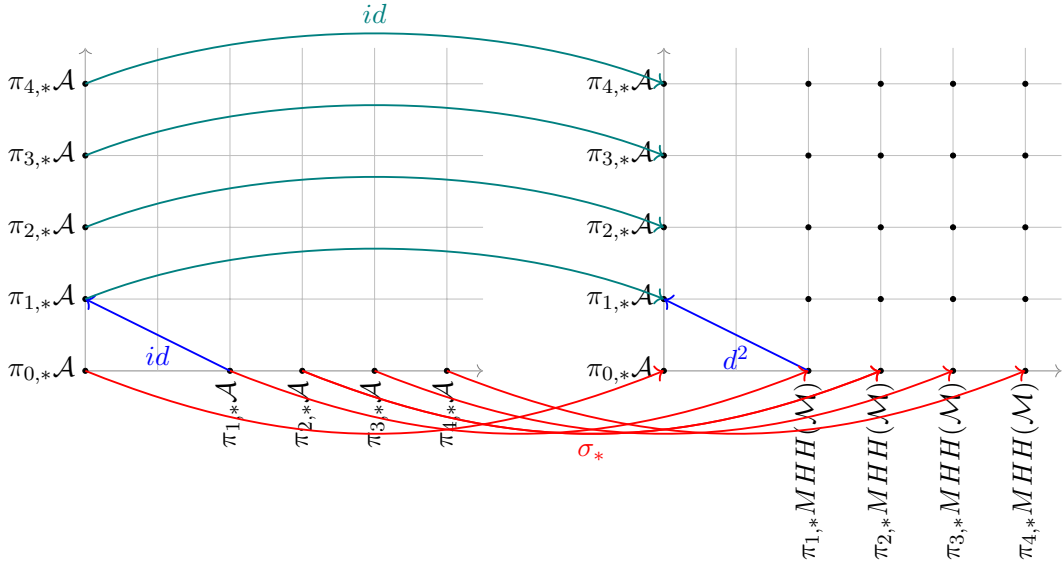

    Proceeding to the $E^q$-page, we encounter a commutative diagram:
    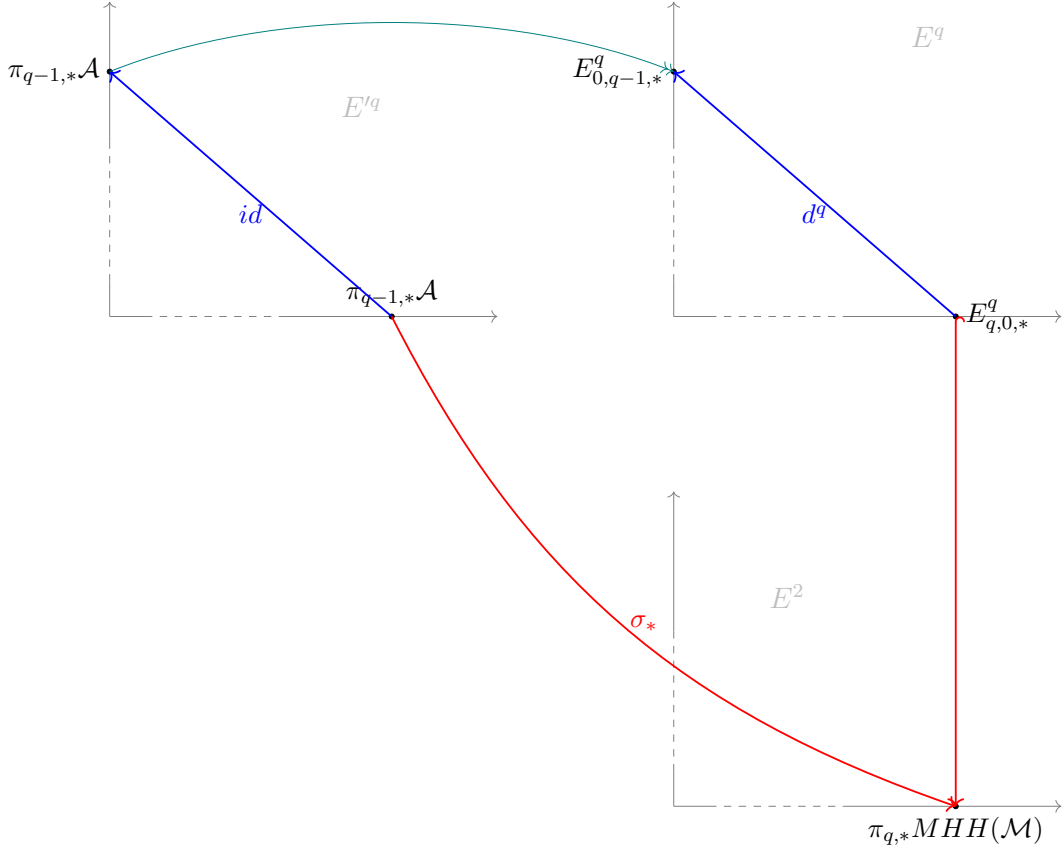
\begin{figure}[H]
    \centering
    \adjustbox{max width= \textwidth}{\begin{tikzpicture}
    \draw[gray] (0,0)--(0,0.5);
    \draw[dashed,gray] (0,0.5)--(0,2.5);
    \draw[->, gray] (0,2.5)--(0,4.5);
    \draw[ gray] (0,0)--(0.5,0);
    \draw[dashed,gray] (0.5,0)--(2.5,0);
    \draw[->, gray] (2.5,0)--(5.5,0);

    \draw[gray] (0+8,-7)--(0+8,-6.5);
    \draw[dashed,gray] (0+8,0.5-7)--(0+8,2.5-7);
    \draw[->, gray] (0+8,2.5-7)--(0+8,4.5-7);
    \draw[ gray] (0+8,0-7)--(0.5+8,0-7);
    \draw[dashed,gray] (0.5+8,0-7)--(2.5+8,0-7);
    \draw[->, gray] (2.5+8,0-7)--(5.5+8,0-7);
    
    \draw[gray] (8,0)--(8,0.5);
    \draw[dashed,gray] (8,0.5)--(8,2.5);
    \draw[->, gray] (8,2.5)--(8,4.5);
    \draw[ gray] (8,0)--(8.5,0);
    \draw[dashed,gray] (8.5,0)--(10.5,0);
    \draw[->, gray] (10.5,0)--(13.5,0);

    \filldraw (0,3.5) circle (1pt);
    \filldraw (4,0) circle (1pt);
    
    \filldraw (8,3.5) circle (1pt);
    \filldraw (12,0) circle (1pt);

    \filldraw (4+8,-7) circle (1pt);
    
    \node at (0,3.5) [left] {\small $\pi_{q-1,*}\mathcal{A}$};
    \node at (8,3.5) [left] {\small $E^q_{0,q-1,*}$};
    \node at (4,0) [above] {\small $\pi_{q-1,*}\mathcal{A}$};
    \node at (4+8,-7) [below, anchor=north] {\small $\pi_{q,*}MHH(\mathcal{M})$};
    \node at (12,0) [below, anchor=west] {\small $E^q_{q,0,*}$};

    \node at (4,3) [anchor=east, lightgray] {$E'^q$};
    \node at (12,4) [anchor=east, lightgray] {$E^q$};
    \node at (2+8,-4) [anchor=east, lightgray] {$E^2$};

    \draw[->, line width=0.75pt, blue] (4,0)--(0,3.5) node[blue, below, midway]{\small $id$};
    \draw[->, line width=0.75pt, blue] (12,0)--(8,3.5) node[blue, below, midway]{\small $d^q$};
    \path (4,0) edge[->, line width=0.75pt, red, bend right=22] node[red, right, midway]{\small $\sigma_*$} (4+8,-7);
    \path (0,3.5) edge[->>, teal, bend left=22, looseness=0.8] (8,3.5);
    \path (12,0) edge[right hook->, line width=0.75pt, red] (4+8,-7);
    
\end{tikzpicture}}
    \caption{The map of spectral sequences at the $E^q$ pages.}

\end{figure}
    that corresponds to the one in \ref{eqn:magic.square} we were looking for.
\end{proof}

This result, as we are going to see, will provide some help in resolving otherwise ambiguous situations. As the homotopy groups of the $\tau$ inverted dual motivic Steenrod algebra:
\[
    \pi_{*,*}\mathcal{A}(p)[\tau^{-1}]\cong 
    \begin{cases}
        \mathbb{F}_2[\tau^{\pm 1}, \tau_i]_{i \geq 0} & \text{ for } p=2\\
            \mathbb{F}_p[\tau^{\pm 1}, \xi_i, \tau_i]_{i \geq 0}/(\tau_i^2) & \text{ for } p=\text{odd}    
    \end{cases}
\]
have a substantially different structure for $p=2$ or $p$ odd, we will distinguish between these two cases.

\subsubsection{\texorpdfstring{$p=2$}{p=2}}\label{subsubs:p=2}

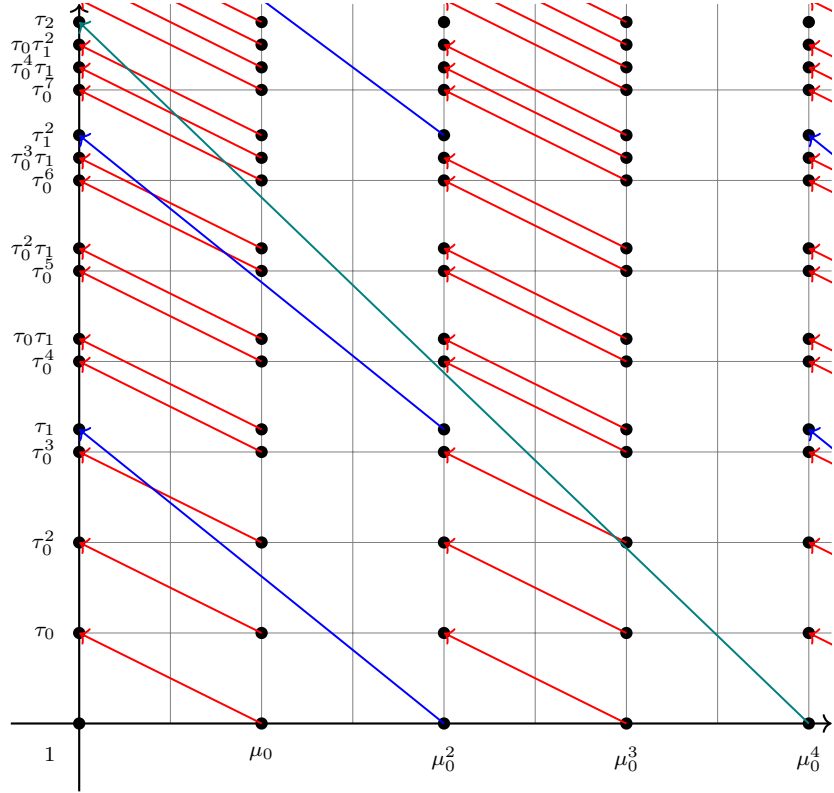
\begin{figure}[H]
\begin{center}
\begin{tikzpicture}[scale=0.6,line width=1pt]

\clip(-3,-1.5) rectangle (16.5,15.9);

\draw[step=2cm ,gray,very thin] (0,0) grid (16.5,15.9);

\draw[black, line width=0.8pt,->](-1.5,0)--(16.5,0);
\draw[black, line width=0.8pt,->] (0,-1.5)--(0,15.9);

\node[label={225:{\scriptsize $1$}}] at ( 0, 0){};
\node[label={180:{\scriptsize $\tau_0$}}] at ( 0, 2){};
\node[label={180:{\scriptsize $\tau_0^2$}}] at ( 0, 4){};
\node[label={180:{\scriptsize $\tau_0^3$}}] at ( 0, 6){};
\node[label={180:{\scriptsize $\tau_1$}}] at ( 0, 6.5){};
\node[label={180:{\scriptsize $\tau_0^4$}}] at ( 0, 8){};
\node[label={180:{\scriptsize $\tau_0\tau_1$}}] at ( 0, 8.5){};
\node[label={180:{\scriptsize $\tau_0^5$}}] at ( 0, 10){};
\node[label={180:{\scriptsize $\tau_0^2\tau_1$}}] at ( 0, 10.5){};
\node[label={180:{\scriptsize $\tau_0^6$}}] at ( 0, 12){};
\node[label={180:{\scriptsize $\tau_0^3\tau_1$}}] at ( 0, 12.5){};
\node[label={180:{\scriptsize $\tau_1^2$}}] at ( 0, 13){};
\node[label={180:{\scriptsize $\tau_0^7$}}] at ( 0, 14){};
\node[label={180:{\scriptsize $\tau_0^4\tau_1$}}] at ( 0, 14.5){};
\node[label={180:{\scriptsize $\tau_0\tau_1^2$}}] at ( 0, 15){};
\node[label={180:{\scriptsize $\tau_2$}}] at ( 0, 15.5){};

\node[label={270:{\scriptsize $\mu_0$}}] at ( 4, 0){};

\node[label={270:{\scriptsize $\mu^2_0$}}] at ( 8, 0){};

\node[label={270:{\scriptsize $\mu^3_0$}}] at ( 12, 0){};

\node[label={270:{\scriptsize $\mu^4_0$}}] at ( 16, 0){};

{\draw[fill]
(0, 0) circle (3pt)
(0, 2) circle (3pt)
(0, 4) circle (3pt)
(0, 6) circle (3pt)
(0, 6.5) circle (3pt)
(0, 8)  circle (3pt)
(0, 8.5) circle (3pt)
(0, 10) circle (3pt)
(0, 10.5) circle (3pt)
(0, 12) circle (3pt)
(0, 12.5) circle (3pt)
(0, 13) circle (3pt)
(0, 14) circle (3pt)
(0, 14.5) circle (3pt)
(0, 15) circle (3pt)
(0, 15.5) circle (3pt)
;}

{\draw[fill]
(4, 0) circle (3pt)
(4, 2) circle (3pt)
(4, 4) circle (3pt)
(4, 6) circle (3pt)
(4, 6.5) circle (3pt)
(4, 8)  circle (3pt)
(4, 8.5) circle (3pt)
(4, 10) circle (3pt)
(4, 10.5) circle (3pt)
(4, 12) circle (3pt)
(4, 12.5) circle (3pt)
(4, 13) circle (3pt)
(4, 14) circle (3pt)
(4, 14.5) circle (3pt)
(4, 15) circle (3pt)
(4, 15.5) circle (3pt)
;}

{\draw[fill]
(8, 0) circle (3pt)
(8, 2) circle (3pt)
(8, 4) circle (3pt)
(8, 6) circle (3pt)
(8, 6.5) circle (3pt)
(8, 8)  circle (3pt)
(8, 8.5) circle (3pt)
(8, 10) circle (3pt)
(8, 10.5) circle (3pt)
(8, 12) circle (3pt)
(8, 12.5) circle (3pt)
(8, 13) circle (3pt)
(8, 14) circle (3pt)
(8, 14.5) circle (3pt)
(8, 15) circle (3pt)
(8, 15.5) circle (3pt)
;}

{\draw[fill]
(12, 0) circle (3pt)
(12, 2) circle (3pt)
(12, 4) circle (3pt)
(12, 6) circle (3pt)
(12, 6.5) circle (3pt)
(12, 8)  circle (3pt)
(12, 8.5) circle (3pt)
(12, 10) circle (3pt)
(12, 10.5) circle (3pt)
(12, 12) circle (3pt)
(12, 12.5) circle (3pt)
(12, 13) circle (3pt)
(12, 14) circle (3pt)
(12, 14.5) circle (3pt)
(12, 15) circle (3pt)
(12, 15.5) circle (3pt)
;}

{\draw[fill]
(16, 0) circle (3pt)
(16, 2) circle (3pt)
(16, 4) circle (3pt)
(16, 6) circle (3pt)
(16, 6.5) circle (3pt)
(16, 8)  circle (3pt)
(16, 8.5) circle (3pt)
(16, 10) circle (3pt)
(16, 10.5) circle (3pt)
(16, 12) circle (3pt)
(16, 12.5) circle (3pt)
(16, 13) circle (3pt)
(16, 14) circle (3pt)
(16, 14.5) circle (3pt)
(16, 15) circle (3pt)
(16, 15.5) circle (3pt)
;}

\draw[colarr2,line width=0.75pt,->] (4, 0) ->(0, 2);
\draw[colarr2,line width=0.75pt,->] (4, 2) --(0,4);
\draw[colarr2,line width=0.75pt,->] (4, 4) --(0,6);
\draw[colarr2,line width=0.75pt,->] (4, 6) --(0,8);
\draw[colarr2,line width=0.75pt,->] (4, 6.5) --(0,8.5);
\draw[colarr2,line width=0.75pt,->] (4, 8)  --(0,10);
\draw[colarr2,line width=0.75pt,->] (4, 8.5) --(0,10.5);
\draw[colarr2,line width=0.75pt,->] (4, 10) --(0,12);
\draw[colarr2,line width=0.75pt,->] (4, 10.5)--(0,12.5);
\draw[colarr2,line width=0.75pt,->] (4, 12)--(0,14); 
\draw[colarr2,line width=0.75pt,->] (4, 12.5)--(0,14.5); 
\draw[colarr2,line width=0.75pt,->] (4, 13)--(0,15); 
\draw[colarr2,line width=0.75pt,->] (4, 14)--(0,16); 
\draw[colarr2,line width=0.75pt,->] (4, 14.5)--(0,16.5); 
\draw[colarr2,line width=0.75pt,->] (4, 15)--(0,17); 
\draw[colarr2,line width=0.75pt,->] (4, 15.5)--(0,17.5);

\draw[colarr2,line width=0.75pt,->] (12, 0) --(8, 2);
\draw[colarr2,line width=0.75pt,->] (12, 2) --(8,4);
\draw[colarr2,line width=0.75pt,->] (12, 4) --(8,6);
\draw[colarr2,line width=0.75pt,->] (12, 6) --(8,8);
\draw[colarr2,line width=0.75pt,->] (12, 6.5) --(8,8.5);
\draw[colarr2,line width=0.75pt,->] (12, 8)  --(8,10);
\draw[colarr2,line width=0.75pt,->] (12, 8.5) --(8,10.5);
\draw[colarr2,line width=0.75pt,->] (12, 10) --(8,12);
\draw[colarr2,line width=0.75pt,->] (12, 10.5)--(8,12.5);
\draw[colarr2,line width=0.75pt,->] (12, 12)--(8,14); 
\draw[colarr2,line width=0.75pt,->] (12, 12.5)--(8,14.5); 
\draw[colarr2,line width=0.75pt,->] (12, 13)--(8,15); 
\draw[colarr2,line width=0.75pt,->] (12, 14)--(8,16); 
\draw[colarr2,line width=0.75pt,->] (12, 14.5)--(8,16.5); 
\draw[colarr2,line width=0.75pt,->] (12, 15)--(8,17); 
\draw[colarr2,line width=0.75pt,->] (12, 15.5)--(8,17.5); 

\draw[colarr2,line width=0.75pt,->] (20, 0) --(16, 2);
\draw[colarr2,line width=0.75pt,->] (20, 2) --(16,4);
\draw[colarr2,line width=0.75pt,->] (20, 4) --(16,6);
\draw[colarr2,line width=0.75pt,->] (20, 6) --(16,8);
\draw[colarr2,line width=0.75pt,->] (20, 6.5) --(16,8.5);
\draw[colarr2,line width=0.75pt,->] (20, 8)  --(16,10);
\draw[colarr2,line width=0.75pt,->] (20, 8.5) --(16,10.5);
\draw[colarr2,line width=0.75pt,->] (20, 10) --(16,12);
\draw[colarr2,line width=0.75pt,->] (20, 10.5)--(16,12.5);
\draw[colarr2,line width=0.75pt,->] (20, 12)--(16,14); 
\draw[colarr2,line width=0.75pt,->] (20, 12.5)--(16,14.5); 
\draw[colarr2,line width=0.75pt,->] (20, 13)--(16,15); 
\draw[colarr2,line width=0.75pt,->] (20, 14)--(16,16); 
\draw[colarr2,line width=0.75pt,->] (20, 14.5)--(16,16.5); 
\draw[colarr2,line width=0.75pt,->] (20, 15)--(16,17); 
\draw[colarr2,line width=0.75pt,->] (20, 15.5)--(16,17.5);

\draw[colarr4,line width=0.75pt,->] (8, 0) --(0, 6.5);
\draw[colarr4,line width=0.75pt,->] (8,6.5) -- (0,13);
\draw[colarr4,line width=0.75pt,->] (8,13) -- (0,19); 
\draw[colarr4,line width=0.75pt,->] (24, 0) --(16, 6.5);
\draw[colarr4,line width=0.75pt,->] (24,6.5) -- (16,13);

\draw[colarr8,line width=0.75pt,->] (16,0) -- (0,15.5);

\end{tikzpicture}

\caption{The spectral sequence for $MHH(M\mathbb{Z}/2)[\tau^{-1}]$ in low degrees.  Each point represents the generator of a module isomorphic to $\mathbb{F}_2[\tau^{\pm 1}]$.}

\end{center}

\end{figure}

We begin by analysing the even-primary case: 

\[
\begin{adjustbox}{max width=\textwidth, center}
$\displaystyle
E^2_{s,t,*}=\pi_{s,*}MHH(M\mathbb{Z}/2)[\tau^{-1}] \otimes_{\mathbb{F}_2[\tau^{\pm 1}]} \pi_{t,*}\mathcal{A}(2)[\tau^{-1}] \Rightarrow \pi_{s+t,*} M\mathbb{Z}/2[\tau^{-1}].
$    
\end{adjustbox}
\]

\begin{rmk}\label{rmk:key.properties.spec.seq}
    We recall some key features of our spectral sequence, of which we will make abundant use.
\begin{itemize}
    \item The $E^{\infty}$ term is non zero only in position $(0,0,*)$.
    \item The spectral sequence is first quadrant. This implies that the $d^q$ differential acts trivially one the modules $E^q_{s,t,*}$ with $s<q$.
    \item In particular, combining the two above, the modules $E^q_{s,t,*}$ with $s<q$, $t<q-1$ and $(s,t) \neq (0,0)$ must be trivial.
    \item The spectral sequence is multiplicative, in the sense that the differentials satisfy a Leibniz rule, which, given our truncation, is weighted on the degree of the classes involved.
\end{itemize}
\end{rmk}

Before carrying out a general induction computation, we analyse explicitly what happens in low degrees to grasp an idea of the behaviour of the spectral sequence and to provide a base step. Aided by the above properties, our study will proceed from the origin towards the right, acquiring at the same time information on elements of a higher degree in $\pi_{*,*}MHH(M\mathbb{Z}/2)[\tau^{-1}]$ and on the behaviour of the spectral sequence in higher pages.

In position $E^2_{0,0,*}$ we have a module isomorphic to $\mathbb{F}_2[\tau^{\pm 1}]$, generated by 1. By the above remarks, it is a permanent cycle, in accordance with the convergence assumption.

As the infinity term of the spectral sequence does not have anything in degree $(1,0,*)$, we can conclude $\pi_{1,*}MHH(M\mathbb{Z}/2)[\tau^{-1}]=0$: any element here would produce a permanent cycle. Given the multiplicative decomposition of the $E^2$-page, we conclude that the whole column $E^2_{1,*,*}\cong 0$ is trivial.

Next, the presence of $\tau_0$ in degree 1 in the dual motivic Steenrod algebra implies that there exists an element $\mu_0 \in \pi_{2,*}MHH(M\mathbb{Z}/2)[\tau^{-1}]$ such that:
\begin{equation*}
d^2(\mu_0)=\tau_0.
\end{equation*}
This extends by multiplicativity to an isomorphism of modules:
\[
    \mathbb{F}_2[\tau^{\pm 1}]\{ \mu_0 \} \xrightarrow{d^2} \mathbb{F}_2[\tau^{\pm 1}]\{ \tau_0 \} \cong E^2_{0,1,*}
\]
As we cannot produce permanent cycles, this has to exhaust all the classes in $E^2_{2,0,*}$; we then conclude the isomorphism:
\begin{equation*}
\pi_{2,*}MHH(M\mathbb{Z}/2)[\tau^{-1}] \cong \mathbb{F}_2[\tau^{\pm 1}]\{ \mu_0 \}.
\end{equation*}

By the decomposition of the $E^2$ page, all elements in $E^2_{2,*,*}$ are of the form $\mu_0 \cdot \alpha$, with $\alpha \in \pi_{*,*}\mathcal{A}(2)[\tau^{\pm 1}]$; as $d^2(\mu_0 \cdot \alpha)=\tau_0 \alpha \neq 0$ for $\alpha \neq 0$, we conclude that $E^k_{2,*,*}=0$ for $k\geq 3$.
On the other hand, any element $x \in E^2_{0,*,*}$ multiple of $\tau_0$ is hit by a 2-differential originating in $\mu_0x/\tau_0$; in other words, all and only the multiples of $\tau_0$ vanish from the zeroth column after the $E^2$-page; more precisely:
\begin{equation} \label{eq:E3_{0,*,*},p=2}
    E^3{0,*,*} \cong E^2{0,*,*}/\tau_0 E^2{0,*,*} \cong \pi_{*,*} \mathcal{A}(2)[\tau^{-1}]/\langle \tau_0 \rangle.
\end{equation}

Observe that the $d^2$ differential has to be trivial on $E^2_{3,0,*}$, as the first column $E^2_{1,*,*}$ is null. As moreover we have that $E^3_{(0,2,*)}=0$, we can conclude that 
\[
     \pi_{3,*}MHH(M\mathbb{Z}/2)[\tau^{-1}] \cong E^3_{3,0,*} \cong 0,
\]
similarly to what happened in degree 1. Because of the decomposition of the $E^2$-page, whole the third column is empty.

In degree four, we encounter the square $\mu_0^2$. It has a trivial $d^2$-differential for characteristic reason and a trivial $d^3$-differential for $E^3_{1,2,*}\cong 0$. On the other hand, we need an element $\mu_1 \in E^4_{4,0,1}$ with a $d^4$-differential $d^4(\mu_1)=\tau_1$. Thanks to corollary \ref{cor:power.operations}, we can identify $\mu_1=\tau^{-1}\mu_0$; as there is an isomorphism of $\mathbb{F}_2[\tau^{\pm 1}]$-modules:
\[
    \mathbb{F}_2[\tau^{\pm 1}]\{ \mu_0^2 \} \xrightarrow{d^4} \mathbb{F}_2[\tau^{\pm 1}]\{ \tau_1 \} \cong E^4_{0,3,*}
\]
we conclude that $E^4_{4,0,*} \cong \mathbb{F}_2[\tau^{\pm 1}]\{ \mu_0^2 \}$, that implies:
\[
    \pi_{4,*}MHH(M\mathbb{Z}/2)[\tau^{-1}] \cong \mathbb{F}_2[\tau^{\pm 1}]\{ \mu_0^2 \}.
\]

Again, in degree 5 one concludes, as in degree 3: $\pi_{5,*}MHH(M\mathbb{Z}/2)[\tau^{-1}] \cong0$.

In degree 6, we encounter the element $\mu_0^3$, that has non-trivial $d^2$-differential 
\begin{equation*}
d^2(\mu_0^3)=d^2(\mu_0 \cdot \mu_0^2) = d^2(\mu_0 ) \mu_0^2+ d^2(\mu_0^2) \mu_0= \tau_0\mu_0^2.
\end{equation*}
Similarly to what happens in the second column, all elements in $E^2_{6,*,*}$ have a non-trivial $d^2$-differential, landing in the fourth column $E^2_{4,*,*}$. These hit all elements multiple of $\tau_0$ there, removing them from the $E^3$ page. Also, from what we described above, we have that:
\begin{equation*}
E^3_{3,2,*}\cong E^4_{2,3,*} \cong E^5_{1,4,*} \cong E^6_{0,5,*} \cong 0
\end{equation*}
Hence:
\begin{equation*}
E^2_{6,*,*}\cong \pi_{4,*}MHH(M\mathbb{Z}/2)[\tau^{-1}] \cong \mathbb{F}_2[\tau^{\pm 1}]\{ \mu_0^3 \}.
\end{equation*}

Moving to the seventh column, we see that the only possible non-trivial differential for an element $\alpha \in E^2_{7,0,*}$ here would be $\alpha \xrightarrow{d^7} \tau_1^2$; however, this would require $\tau_1^2$ to survive up to the $E^7$ page. This is incompatible with the $d^4$ differential arising from the Leibniz rule $\mu_0^2\tau_1 \xrightarrow{d^4} \tau_1^2$, that makes $\tau_1^2$ vanish after the $E^4$ page; so we can conclude that $E^2_{7,0,*} \cong 0$, and hence $E^2_{7,*,*}\cong 0$ 

Now, the elements in the fourth column that survive to the $E^3$-page (actually, to the $E^4$ page) come with a $d^4$ differential landing in the zeroth column given by the Leibniz rule, namely:
\begin{equation*}
d^4(\mu_0^2\beta)=\tau_1\beta
\end{equation*}
In fact, this $\tau_1 \beta \in E^4_{0,*,*}$ is nonzero for $0 \neq \beta \in E^4_{0,*,*}$, as from the description of $E^4_{0,*,*}$ (there are no $d^3$ differentials hitting the zeroth column, so $E^4_{0,*,*}\cong E^3_{0,*,*}$) it is just a polynomial algebra. As there are no differentials involving the fourth column that are shorter than a $d^4$ but the $d^2$ differentials from the sixth column we already took into the account, we can conclude the necessity and the non-triviality of these $d^4$ differentials. This mechanism involves all the non-zero elements of $E^4_{4,*,*}$, so we conclude that $E^5_{4,*,*}\cong 0$. Observe incidentally that any element of the form $\tau_1\beta \in E^4_{0,*,*}$ is hit by one of these $d^4$ differentials, so there are no elements containing $\tau_1$ in the zeroth column in pages higher than $E^5$.

Following this pattern, we obtain the following. 

\begin{prop}\label{prop:pi.MHH(MZ/2)[tau-1]}
    $\pi_{*,*}MHH(M\mathbb{Z}/2)[\tau^{-1}]$ is a polynomial ring in a generator $\mu_0$ of degree $(2,0)$; the differential behaviour is determined by the largest power of 2 dividing the exponent: 
    \begin{equation}\label{eq:diff.MHH.p=2}
        \mu_0^{n2^i} \xrightarrow{d_{2^{i+i}}} \mu_0^{(n-1)2^i}\tau_{i} \text{, with } n \text{ odd.}
    \end{equation}
\end{prop}

One can deduce from this a very precise description of all differentials in the spectral sequence, just by applying the Leibniz rule:
\begin{cor}\label{cor:overwhelming.p=2}
The elements in $E^2_{*,*,*}$ 
are given by $\mathbb{F}_2[\tau^{\pm 1}]$-linear combinations of generators of the form $\mu_0^{n2^i}\tau_J^{e_J}$, where $J=(j_1,j_2,\ldots,j_M)$ and $e_J=(e_{j_1},e_{j_2},\ldots,e_{j_M})$ are multi-indices of non-negative integers, with $n$ odd, the $j_h$ distinct and the $e_j$ strictly positive, and  $\xi_J^{e_J}=\xi_{j_1}^{e_{j_1}}\xi_{j_2}^{e_{j_2}}\cdots \xi_{j_M}^{e_{j_M}}$. The degrees are: $|\mu_0|=(2,0,0)$ and $|\tau_j|=(0,2^{j+1}-1,2^j-1)$.

$\mu_0^{n2^i}\tau_J^{e_J}$ survives to the $E^{2^{k+1}}$ page and is involved in a $d^{2^{k+1}}$ differential, where $k=min(i,J)$. The differential is exiting if $k=i$ and has image $\mu_0^{(n-1)2^i}\tau_i\tau_J^{e_J}$, otherwise it is entering from $\mu^{n2^i+2^k}\tau_J^{e_J}/\tau_k$.
\end{cor}

\begin{proof}[Proof of \ref{prop:pi.MHH(MZ/2)[tau-1]}]
We work this out by induction on a natural number $k\geq 1$, proving that there is an isomorphism of graded rings:
\begin{equation}\label{eq:isomorphism.pi.MHH.2}
    \pi_{*,*}MHH(M\mathbb{Z}/2)[\tau^{-1}] \cong \mathbb{F}_2[\tau^{\pm 1}, \mu_0]
\end{equation}
    
with $|\tau|=(0,-1)$, $|\mu_0|=(2,0)$, for the first degree $s \leq 2^{k+1}-1$ and all weights. At the same time, we show that the behaviour of these elements in the spectral sequence is determined by the rules expressed in our statement, plus the Leibniz rule, in the same range. As we will see, this will determine uniquely the behaviour in the spectral sequence of all elements in the columns $E^{2}_{s,*,*}$ for $1 \leq s \leq 2^{k+1}-1$; in particular, all modules in this range will be trivial after the $E^{2^k}$ page. Moreover, we will prove that $E^{2^k+1}_{0,*,*}\cong E^2_{0,*,*}/\langle \tau_0, \ldots \tau_{k-1}\rangle$.

Some base steps were already performed above.

Suppose now our claim holds for a certain positive number $k$; as all columns $E^{2^k+1}_{s,*,*}$ are trivial for $1 \leq s \leq 2^{k+1}-1$, we conclude that there are isomorphisms:
\[
    E^{2^k+1}_{0,*,*}\cong E^{2^k+2}_{0,*,*}\cong \ldots \cong E^{2^{k+1}}_{0,*,*}\cong E^2_{0,*,*}/\langle \tau_0, \ldots \tau_{k-1}\rangle.
\]
In particular, the element $\tau_k$, with $|\tau_k|=(2^{k+1}-1, 2^k-1)$, generates the non-trivial module in the zeroth column of the $E^{2^{k+1}}$-page with smallest positive degree. By the convergence of the spectral sequence, we must then have an element $\mu_k \in E^2_{2^{k+1},0,2^k-1}$ surviving up to the $E^{2^{k+1}}$-page and supporting a nontrivial $d^{2^{k+1}}$-differential $d^{2^{k+1}}(\mu_k)=\tau_k$. Using corollary \ref{cor:power.operations}, we obtain: $\mu_k\cong \tau^{-2^{k}+1}\mu_0^{2^k}$.

The $d^{2^{k+1}}$ differential extends linearly to an isomorphism of modules:
\[
    \mathbb{F}_2[\tau^{\pm 1}]\{ \mu_0^{2^k} \} \xrightarrow{d^{2^{k+1}}} \mathbb{F}_2[\tau^{\pm 1}]\{ \tau_k \} \cong E^{2^{k+1}}_{0,2^{k+1}-1,*}
\]
Because of the vanishing result on the columns $E^{2}_{s,*,*}$ for $1 \leq s \leq 2^{k+1}-1$, we conclude no other element can be found in $E^2_{2^{k+1},0,*}$, as we must avoid permanent cycles. Hence the isomorphism of equation \ref{eq:isomorphism.pi.MHH.2} extends to degree $2^{k+1}$.

\begin{rmk} \label{rmk:no.diffs.out.of.E2{2k+1-x,*,*}}
    Observe incidentally that there is a more concrete reason why no non-trivial differential shorter than a $d^{2^{k+1}}$ can come out of the $E^2_{2^{k+1},*,*}$ column: from our description, given a number $2 \leq x \leq 2^{k+1}$, the elements in the column $E^2_{2^{k+1}-x,*,*}$ vanish before the $E^x$ page or support in the $E^{x}$ page non-trivial $d^x$-differentials, that are isomorphisms onto their image. Anything else coming from the $2^{k+1}$-th column would be incompatible with this picture.
\end{rmk}

The rest of the proof of the induction step follows by application of the Leibniz rule. More precisely, one proceeds by induction on $i=0, \ldots, 2^{k+1}-1$ to show that the structure of $E^2_{2^{k+1}+i,0,*}$ coincides with that given by the powers of $\mu_0$, and that these elements support nontrivial differentials coming from the Leibniz rule. In other words, this means that if $i$ is odd $E^2_{2^{k+1}+i,0,*}$ is a trivial module, while if $i$ is even and $2^{j+1}$ is the largest power of $2$ dividing $i$, $E^2_{2^{k+1}+i,0,*}$ is generated by $\mu_0^{2^k+i/2}$, which survives until the $E^{2^{j+1}}$ page and supports a non-trivial $d^{2^{j+1}}$ differential, which extends to an isomorphism of $\mathbb{F}_2[\tau^{\pm 1}]$ modules. The step $i=0$ was discussed above; suppose then that the thesis holds up to a certain $i-1$.

If $i$ is odd, one sees that no element can be introduced, as they would produce permanent cycles:
\begin{itemize}
    \item in column $0$, the first ``non assigned'' module is generated by $\tau_k^2$ in degree $2^{k+2}-2$; now, $\mu_0^{2^k}\tau_k$ supports a $d^{2^{k+1}}$-differential hitting this element in the  $E^{2^{k+1}}$-page, by the Leibniz rule. This excludes the possibility that $\tau_k^2$ survives to higher pages, so no element in the region we are studying will support a non-trivial differential hitting it (as it would be longer than a $d^{2^{k+1}}$). In particular, no element has to be introduced in $\pi_{*, *}MHH(M\mathbb{Z}/2)[\tau^{-1}]$ in the region we are analysing to support transgressive differentials.
    \item the general induction procedure also allows us to exclude any differential landing in the columns from 1 to $2^k-1$, as the structure there is completely determined and any other differential would only produce cycles. This can follow also from what is stated in remark \ref{rmk:no.diffs.out.of.E2{2k+1-x,*,*}}.
    \item the remaining part of the proof finally excludes differentials starting from the horizontal axis in odd degree with image in the region $2^{k+1} \leq s \leq 2^{k+2}-1$: differentials out of powers of $\mu_0$ (forced by the Leibniz rule) must be non-trivial and kill enough elements in low degrees not to give space to any other non-trivial differential (hence any other element).
\end{itemize}

At any even $i$, let, as before, $2^{j+1}$ be the largest power of $2$ dividing $i$. Then the Leibniz rule imposes a differential 
\[
    d^{2^{j+1}}(\mu_0^{2^k+i})= \mu_0^{2^k+i/2-2^j} \tau_j.
\]
This differential is non-trivial. In fact, given the structure in the lower degrees, the element $\mu_0^{2^k+i/2-2^j} \tau_j$ is nonzero in the $E^{2^{j+1}}$ page: observe that $2^{j+1}$ divides $2^k+i/2-2^j$; so on one hand all lower degree differentials are trivial on $\mu_0^{2^k+i/2-2^j} \tau_j$ by the Leibniz rule (because they are trivial on $\mu_0^{2^k+i/2-2^j}$), and on the other hand, no lower degree differential applied to powers of $\mu_0$ has $\tau_j$ in the image (here we use that by induction hypothesis in lower degrees there are only the powers of $\mu_0$). This implies in particular that the differential extends to an isomorphism of modules:
\[
    \mathbb{F}_2[\tau^{\pm 1}]\{ \mu_0^{2^k+i/2} \} \xrightarrow{d^{2^{j+1}}} \mathbb{F}_2[\tau^{\pm 1}]\{ \mu_0^{2^k+i/2-2^j} \tau_j \} 
\]
implying in particular that $\mathbb{F}_2[\tau^{\pm 1}]\{ \mu_0^{2^k+i/2} \}$ is a free $\mathbb{F}_2[\tau^{\pm 1}]$-module.  Now we make the following observation:

\begin{claim} \label{clm:in.the.proof.of.prop:pi.MHH(MZ/2)[tau-1]}
    Suppose $\mu_0^{m2^{j+1}+2^j}$ for some $m>0$ generates a free $\mathbb{F}_2[\tau^{\pm 1}]$-module and has trivial differentials up to $d^{2^{j+1}}$.
    
    Then all the products $\mu_0^{m2^{j+1}+2^j+h}$ generate a free $\mathbb{F}_2[\tau^{\pm 1}]$-module for $0 \leq h \leq 2^j-1$; these elements support in the spectral sequence the non-trivial differentials determined by the Leibniz rule; moreover, there is an isomorphism of $E^{2^{j+1}}_{0,*,*}$-modules: 
    \[
        E^{2^{j+1}}_{0,*,*} \xrightarrow{\cdot \mu_0^{m2^{j+1}+2^j}} E^{2^{j+1}}_{m2^{j+2}+2^{j+1},*,*}
    \]
\end{claim}

Hence, the $d^{2^{j+1}}$ differential that acts non-trivially on $\mu_0^{m2^{j+1}+2^j}$, extended to the whole column by the Leibniz rule, provides an isomorphism of graded modules onto its image, corresponding to the $E^{2^{j+1}}_{0,*,*}$-module of $E^{2^{j+1}}_{2^{k+1}+i-2^{j+1},*,*}$ generated by $\tau_j$ (it kills all multiples of $\tau_j$ in that column). Observe that this whipes out the whole column over $\mu_0^{m2^{j+1}+2^j}$: $E^{2^{j+1}+1}_{m2^{j+2}+2^{j+1},*,*} \cong 0$.

At the end of the induction procedure, we see that there is an isomorphism of  $E^{2^{k+1}}_{0,*,*}$-modules:
\[
    E^{2^{k+1}}_{2^{k+1},*,*} \cong E^{2^{k+1}}_{0,*,*} \cdot \mu_0^{2^k},
\]
so everything that remains in $E^{2^{k+1}}_{2^{k+1},*,*}$ supports non-trivial differentials following the Leibniz rule. This implies in particular that the entire column $E^{2^{k+1}+1}_{2^{k+1},*,*}\cong 0$ is trivial; this concludes the $k+1$ step of the induction procedure, as we have uniquely determined the structure of $\pi_{*,*}MHH(M\mathbb{Z}/2)[\tau^{-1}]$ up to first degree $2^{k+2}-1$ and the behaviour of the spectral sequence of elements in $E^2_{s,*,*}$ for $1 \leq s \leq 2^{k+2}-1$.
\end{proof}

Observe that in this spectral sequence, we can deduce that all pages decompose as a tensor product of an algebra on the horizontal axis and an algebra on the vertical axis. This will not be the case when considering an odd prime.

\begin{proof}[Proof of \ref{clm:in.the.proof.of.prop:pi.MHH(MZ/2)[tau-1]}]
    By induction on $0 \leq j \leq k$. 

    The case $j=0$ corresponds to odd powers of $\mu_0$: here we have nothing to prove.

    Suppose the claim holds for all indices up to a certain $j-1$. Given $\mu_0^{m2^{j+1}+2^j}$, we make a finite induction on $0 \leq n \leq j-1$. At each step, we consider the element  $\mu_0^{m2^{j+1}+2^j+2^n}$; by the previous steps of this induction on $n$ we know that the differential given by the Leibniz rule (a $d^{2^{n+1}}$) provides an isomorphism of $\mathbb{F}_2[\tau^{\pm 1}]$-modules
    \[
        \mathbb{F}_2[\tau^{\pm 1}]\{\mu_0^{m2^{j+1}+2^j+2^n}\} \xrightarrow{d^{2^{n+1}}} \mathbb{F}_2[\tau^{\pm 1}]\{\mu_0^{m2^{j+1}+2^j} \tau_j\}
    \]
    Furthermore, this allows to apply the claim for $j=n$; in particular, 
    \[
        E^{2^{n+1}}_{0,*,*} \xrightarrow{\cdot \mu_0^{m2^{j+1}+2^j+2^n}} E^{2^{n+1}}_{m2^{j+2}+2^{j+1}+2^{n+1},*,*}
    \]
    is an isomorphism. Hence by linearity the $d^{2^{n+1}}$ extends to an isomorphism:
    \[
        E^{2^{n+1}}_{m2^{j+2}+2^{j+1}+2^{n+1},*,*}\xrightarrow{d^{2^{n+1}}} \tau_j \cdot E^{2^{n+1}}_{m2^{j+2}+2^{j+1},*,*}
    \]
    As each of these differentials acts as the analogous hitting the zeroth column, we get the desired isomorphism in the $E^{2^{j+1}}$-pages.
\end{proof}

\subsubsection{\texorpdfstring{Odd $p$}{Odd p}}\label{subsubs:odd.p}

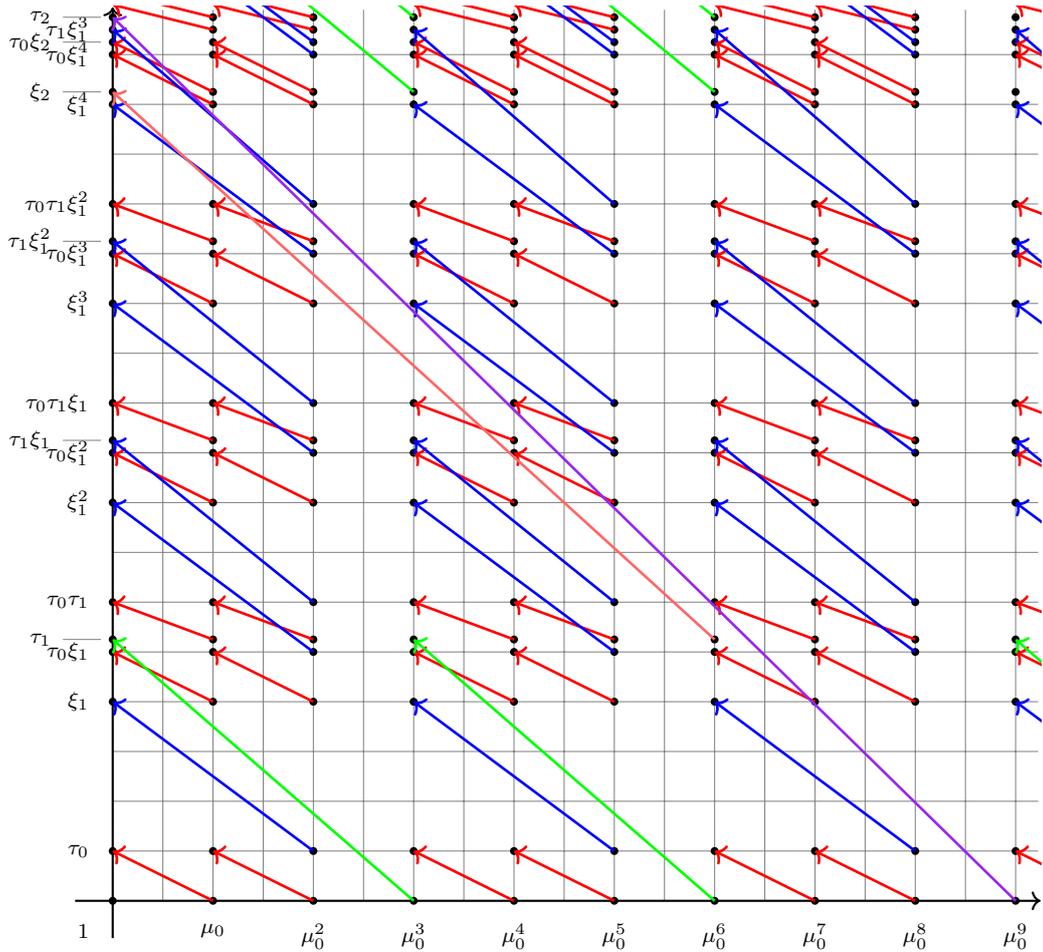
\begin{figure}[H]
\begin{center}
\begin{tikzpicture}[scale=0.33,line width=1pt] 

\clip(-5,-2) rectangle (37,35.95);

\draw[step=2cm ,gray,very thin] (0,0) grid (36.9,35.8);

\draw[black, line width=0.8pt,->](-1.5,0)--(37,0);
\draw[black, line width=0.8pt,->] (0,-1.5)-- (0,35.9);

\node[label={225:{\scriptsize $1$}}] at ( 0, 0){};

\node[pin={[pin distance= 0cm]180:{\scriptsize $\tau_0$}}] at ( 0, 2){};

\node[pin={[pin distance= 0cm]180:{\scriptsize $\xi_1$}}] at ( 0, 8){};

\node[pin={[pin distance= 0cm]180:{\scriptsize $\tau_0\xi_1$}}] at ( 0, 10){};
\node[pin={[pin distance= 0.5cm]180:{\scriptsize $\tau_1$}}] at ( 0, 10.5){};

\node[pin={[pin distance= 0cm]180:{\scriptsize $\tau_0\tau_1$}}] at ( 0, 12){};

\node[pin={[pin distance= 0cm]180:{\scriptsize $\xi_1^2$}}] at ( 0, 16.0){};

\node[pin={[pin distance= 0cm]180:{\scriptsize $\tau_0\xi_1^2$}}] at ( 0, 18.0){};
\node[pin={[pin distance= 0.5cm]180:{\scriptsize $\tau_1\xi_1$}}] at ( 0, 18.5){};

\node[pin={[pin distance= 0cm]180:{\scriptsize $\tau_0\tau_1\xi_1$}}] at ( 0, 20.0){};

\node[pin={[pin distance= 0cm]180:{\scriptsize $\xi_1^3$}}] at ( 0, 24){};

\node[pin={[pin distance= 0cm]180:{\scriptsize $\tau_0\xi_1^3$}}] at ( 0, 26.0){};
\node[pin={[pin distance= 0.5cm]180:{\scriptsize $\tau_1\xi_1^2$}}] at ( 0, 26.5){};

\node[pin={[pin distance= 0cm]180:{\scriptsize $\tau_0\tau_1\xi_1^2$}}] at ( 0, 28.0){};

\node[pin={[pin distance= 0cm]180:{\scriptsize $\xi_1^4$}}] at ( 0, 32.0){};
\node[pin={[pin distance= 0.5cm]180:{\scriptsize $\xi_2$}}] at ( 0, 32.5){};

\node[pin={[pin distance= 0cm]180:{\scriptsize $\tau_0\xi_1^4$}}] at ( 0, 34.0){};
\node[pin={[pin distance= 0.5cm]180:{\scriptsize $\tau_0\xi_2$}}] at ( 0, 34.5){};
\node[pin={[pin distance= 0cm]180:{\scriptsize $\tau_1\xi_1^3$}}] at ( 0, 35.0){};
\node[pin={[pin distance= 0.5cm]180:{\scriptsize $\tau_2$}}] at ( 0, 35.5){};

\node[label={270:{\scriptsize $\mu_0$}}] at ( 4, 0){};

\node[label={270:{\scriptsize $\mu_0^2$}}] at ( 8, 0){};

\node[label={270:{\scriptsize $\mu_0^3$}}] at ( 12, 0){};

\node[label={270:{\scriptsize $\mu_0^4$}}] at ( 16.0, 0){};

\node[label={270:{\scriptsize $\mu_0^5$}}] at ( 20, 0){};

\node[label={270:{\scriptsize $\mu_0^6$}}] at ( 24, 0){};

\node[label={270:{\scriptsize $\mu_0^7$}}] at ( 28, 0){};

\node[label={270:{\scriptsize $\mu_0^8$}}] at ( 32, 0){};

\node[label={270:{\scriptsize $\mu_0^9$}}] at ( 36, 0){};

{\draw[fill]
(0, 0) circle (3pt)
(0, 2) circle (3pt)
(0, 8) circle (3pt)
(0, 10) circle (3pt)
(0, 10.5) circle (3pt)
(0, 12)  circle (3pt)
(0, 16.0) circle (3pt)
( 0, 18.0) circle (3pt)
( 0, 18.5) circle (3pt)
( 0, 20.0) circle (3pt)
( 0, 24) circle (3pt)
( 0, 26.0) circle (3pt)
( 0, 26.5) circle (3pt)
( 0, 28.0) circle (3pt)
( 0, 32.0) circle (3pt)
( 0, 32.5) circle (3pt)
( 0, 34.0) circle (3pt)
( 0, 34.5) circle (3pt)
( 0, 35.0) circle (3pt)
( 0, 35.5) circle (3pt)
;}

{\draw[fill]
(4, 0) circle (3pt)
(4, 2) circle (3pt)
(4, 8) circle (3pt)
(4, 10) circle (3pt)
(4, 10.5) circle (3pt)
(4, 12)  circle (3pt)
(4, 16.0) circle (3pt)
(4, 18.0) circle (3pt)
(4, 18.5) circle (3pt)
(4, 20.0) circle (3pt)
(4, 24) circle (3pt)
(4, 26.0) circle (3pt)
(4, 26.5) circle (3pt)
(4, 28.0) circle (3pt)
(4, 32.0) circle (3pt)
(4, 32.5) circle (3pt)
(4, 34.0) circle (3pt)
(4, 34.5) circle (3pt)
(4, 35.0) circle (3pt)
(4, 35.5) circle (3pt)
;}

{\draw[fill]
(8, 0) circle (3pt)
(8, 2) circle (3pt)
(8, 8) circle (3pt)
(8, 10) circle (3pt)
(8, 10.5) circle (3pt)
(8, 12)  circle (3pt)
(8, 16.0) circle (3pt)
(8, 18.0) circle (3pt)
(8, 18.5) circle (3pt)
(8, 20.0) circle (3pt)
(8, 24) circle (3pt)
(8, 26.0) circle (3pt)
(8, 26.5) circle (3pt)
(8, 28.0) circle (3pt)
(8, 32.0) circle (3pt)
(8, 32.5) circle (3pt)
(8, 34.0) circle (3pt)
(8, 34.5) circle (3pt)
(8, 35.0) circle (3pt)
(8, 35.5) circle (3pt)
;}

{\draw[fill]
(12, 0) circle (3pt)
(12, 2) circle (3pt)
(12, 8) circle (3pt)
(12, 10) circle (3pt)
(12, 10.5) circle (3pt)
(12, 12)  circle (3pt)
(12, 16.0) circle (3pt)
(12, 18.0) circle (3pt)
(12, 18.5) circle (3pt)
(12, 20.0) circle (3pt)
(12, 24) circle (3pt)
(12, 26.0) circle (3pt)
(12, 26.5) circle (3pt)
(12, 28.0) circle (3pt)
(12, 32.0) circle (3pt)
(12, 32.5) circle (3pt)
(12, 34.0) circle (3pt)
(12, 34.5) circle (3pt)
(12, 35.0) circle (3pt)
(12, 35.5) circle (3pt)
;}

{\draw[fill]
(16.0, 0) circle (3pt)
(16.0, 2) circle (3pt)
(16.0, 8) circle (3pt)
(16.0, 10) circle (3pt)
(16.0, 10.5) circle (3pt)
(16.0, 12)  circle (3pt)
(16.0, 16.0) circle (3pt)
(16.0, 18.0) circle (3pt)
(16.0, 18.5) circle (3pt)
(16.0, 20.0) circle (3pt)
(16.0, 24) circle (3pt)
(16.0, 26.0) circle (3pt)
(16.0, 26.5) circle (3pt)
(16.0, 28.0) circle (3pt)
(16.0, 32.0) circle (3pt)
(16.0, 32.5) circle (3pt)
(16.0, 34.0) circle (3pt)
(16.0, 34.5) circle (3pt)
(16.0, 35.0) circle (3pt)
(16.0, 35.5) circle (3pt)
;}

{\draw[fill]
(20.0, 0) circle (3pt)
(20.0, 2) circle (3pt)
(20.0, 8) circle (3pt)
(20.0, 10) circle (3pt)
(20.0, 10.5) circle (3pt)
(20.0, 12)  circle (3pt)
(20.0, 16.0) circle (3pt)
(20.0, 18.0) circle (3pt)
(20.0, 18.5) circle (3pt)
(20.0, 20.0) circle (3pt)
(20.0, 24) circle (3pt)
(20.0, 26.0) circle (3pt)
(20.0, 26.5) circle (3pt)
(20.0, 28.0) circle (3pt)
(20.0, 32.0) circle (3pt)
(20.0, 32.5) circle (3pt)
(20.0, 34.0) circle (3pt)
(20.0, 34.5) circle (3pt)
(20.0, 35.0) circle (3pt)
(20.0, 35.5) circle (3pt)
;}

{\draw[fill]
(24.0, 0) circle (3pt)
(24.0, 2) circle (3pt)
(24.0, 8) circle (3pt)
(24.0, 10) circle (3pt)
(24.0, 10.5) circle (3pt)
(24.0, 12)  circle (3pt)
(24.0, 16.0) circle (3pt)
(24.0, 18.0) circle (3pt)
(24.0, 18.5) circle (3pt)
(24.0, 20.0) circle (3pt)
(24.0, 24) circle (3pt)
(24.0, 26.0) circle (3pt)
(24.0, 26.5) circle (3pt)
(24.0, 28.0) circle (3pt)
(24.0, 32.0) circle (3pt)
(24.0, 32.5) circle (3pt)
(24.0, 34.0) circle (3pt)
(24.0, 34.5) circle (3pt)
(24.0, 35.0) circle (3pt)
(24.0, 35.5) circle (3pt)
;}

{\draw[fill]
(28.0, 0) circle (3pt)
(28.0, 2) circle (3pt)
(28.0, 8) circle (3pt)
(28.0, 10) circle (3pt)
(28.0, 10.5) circle (3pt)
(28.0, 12)  circle (3pt)
(28.0, 16.0) circle (3pt)
(28.0, 18.0) circle (3pt)
(28.0, 18.5) circle (3pt)
(28.0, 20.0) circle (3pt)
(28.0, 24) circle (3pt)
(28.0, 26.0) circle (3pt)
(28.0, 26.5) circle (3pt)
(28.0, 28.0) circle (3pt)
(28.0, 32.0) circle (3pt)
(28.0, 32.5) circle (3pt)
(28.0, 34.0) circle (3pt)
(28.0, 34.5) circle (3pt)
(28.0, 35.0) circle (3pt)
(28.0, 35.5) circle (3pt)
;}

{\draw[fill]
(32.0, 0) circle (3pt)
(32.0, 2) circle (3pt)
(32.0, 8) circle (3pt)
(32.0, 10) circle (3pt)
(32.0, 10.5) circle (3pt)
(32.0, 12)  circle (3pt)
(32.0, 16.0) circle (3pt)
(32.0, 18.0) circle (3pt)
(32.0, 18.5) circle (3pt)
(32.0, 20.0) circle (3pt)
(32.0, 24) circle (3pt)
(32.0, 26.0) circle (3pt)
(32.0, 26.5) circle (3pt)
(32.0, 28.0) circle (3pt)
(32.0, 32.0) circle (3pt)
(32.0, 32.5) circle (3pt)
(32.0, 34.0) circle (3pt)
(32.0, 34.5) circle (3pt)
(32.0, 35.0) circle (3pt)
(32.0, 35.5) circle (3pt)
;}

{\draw[fill]
(36.0, 0) circle (3pt)
(36.0, 2) circle (3pt)
(36.0, 8) circle (3pt)
(36.0, 10) circle (3pt)
(36.0, 10.5) circle (3pt)
(36.0, 12)  circle (3pt)
(36.0, 16.0) circle (3pt)
(36.0, 18.0) circle (3pt)
(36.0, 18.5) circle (3pt)
(36.0, 20.0) circle (3pt)
(36.0, 24) circle (3pt)
(36.0, 26.0) circle (3pt)
(36.0, 26.5) circle (3pt)
(36.0, 28.0) circle (3pt)
(36.0, 32.0) circle (3pt)
(36.0, 32.5) circle (3pt)
(36.0, 34.0) circle (3pt)
(36.0, 34.5) circle (3pt)
(36.0, 35.0) circle (3pt)
(36.0, 35.5) circle (3pt)
;}

\draw[colarr2,->] (4, 0) -- (0, 2); 
\draw[colarr2,->] (4, 8)  -- (0,10);
\draw[colarr2,->] (4, 10.5)-- (0,12);
\draw[colarr2,->] (4, 16.0)-- (0,18);
\draw[colarr2,->] (4, 18.5)-- (0,20);
\draw[colarr2,->] (4, 24)-- (0,26);
\draw[colarr2,->] (4, 26.5)-- (0,28);
\draw[colarr2,->] (4, 32)-- (0,34);
\draw[colarr2,->] (4, 32.5)-- (0,34.5);
\draw[colarr2,->] (4, 35)-- (0,36);
\draw[colarr2,->] (4, 35.5)-- (0,36.5);

\draw[colarr2,->] (8, 0) --(4, 2); 
\draw[colarr2,->] (8, 8)  --(4,10);
\draw[colarr2,->] (8, 10.5)--(4,12);
\draw[colarr2,->] (8, 16.0)--(4,18);
\draw[colarr2,->] (8, 18.5)--(4,20);
\draw[colarr2,->] (8, 24)--(4,26);
\draw[colarr2,->] (8, 26.5)--(4,28);
\draw[colarr2,->] (8, 32)--(4,34);
\draw[colarr2,->] (8, 32.5)--(4,34.5);
\draw[colarr2,->] (8, 35)--(4,36);
\draw[colarr2,->] (8, 35.5)--(4,36.5);

\draw[colarr2,->] (16, 0) --(12, 2); 
\draw[colarr2,->] (16, 8)  --(12,10);
\draw[colarr2,->] (16, 10.5)--(12,12);
\draw[colarr2,->] (16, 16.0)--(12,18);
\draw[colarr2,->] (16, 18.5)--(12,20);
\draw[colarr2,->] (16, 24)--(12,26);
\draw[colarr2,->] (16, 26.5)--(12,28);
\draw[colarr2,->] (16, 32)--(12,34);
\draw[colarr2,->] (16, 32.5)--(12,34.5);
\draw[colarr2,->] (16, 35)--(12,36);
\draw[colarr2,->] (16, 35.5)--(12,36.5);

\draw[colarr2,->] (20, 0) --(16, 2); 
\draw[colarr2,->] (20, 8)  --(16,10);
\draw[colarr2,->] (20, 10.5)--(16,12);
\draw[colarr2,->] (20, 16.0)--(16,18);
\draw[colarr2,->] (20, 18.5)--(16,20);
\draw[colarr2,->] (20, 24)--(16,26);
\draw[colarr2,->] (20, 26.5)--(16,28);
\draw[colarr2,->] (20, 32)--(16,34);
\draw[colarr2,->] (20, 32.5)--(16,34.5);
\draw[colarr2,->] (20, 35)--(16,36);
\draw[colarr2,->] (20, 35.5)--(16,36.5);

\draw[colarr2,->] (28, 0) --(24, 2); 
\draw[colarr2,->] (28, 8)  --(24,10);
\draw[colarr2,->] (28, 10.5)--(24,12);
\draw[colarr2,->] (28, 16.0)--(24,18);
\draw[colarr2,->] (28, 18.5)--(24,20);
\draw[colarr2,->] (28, 24)--(24,26);
\draw[colarr2,->] (28, 26.5)--(24,28);
\draw[colarr2,->] (28, 32)--(24,34);
\draw[colarr2,->] (28, 32.5)--(24,34.5);
\draw[colarr2,->] (28, 35)--(24,36);
\draw[colarr2,->] (28, 35.5)--(24,36.5);

\draw[colarr2,->] (32, 0) --(28, 2); 
\draw[colarr2,->] (32, 8)  --(28,10);
\draw[colarr2,->] (32, 10.5)--(28,12);
\draw[colarr2,->] (32, 16.0)--(28,18);
\draw[colarr2,->] (32, 18.5)--(28,20);
\draw[colarr2,->] (32, 24)--(28,26);
\draw[colarr2,->] (32, 26.5)--(28,28);
\draw[colarr2,->] (32, 32)--(28,34);
\draw[colarr2,->] (32, 32.5)--(28,34.5);
\draw[colarr2,->] (32, 35)--(28,36);
\draw[colarr2,->] (32, 35.5)--(28,36.5);

\draw[colarr2,->] (40, 0) --(36, 2); 
\draw[colarr2,->] (40, 8)  --(36,10);
\draw[colarr2,->] (40, 10.5)--(36,12);
\draw[colarr2,->] (40, 16.0)--(36,18);
\draw[colarr2,->] (40, 18.5)--(36,20);
\draw[colarr2,->] (40, 24)--(36,26);
\draw[colarr2,->] (40, 26.5)--(36,28);
\draw[colarr2,->] (40, 32)--(36,34);
\draw[colarr2,->] (40, 32.5)--(36,34.5);
\draw[colarr2,->] (40, 35)--(36,36);
\draw[colarr2,->] (40, 35.5)--(36,36.5);

\draw[colarr4,->] (8, 2) -- (0, 8); 
\draw[colarr4,->] (8,10) -- (0,16.0);
\draw[colarr4,->] (8,12) -- (0,18.5); 
\draw[colarr4,->] (8,18) -- (0,24); 
\draw[colarr4,->] (8,20) -- (0,26.5);  
\draw[colarr4,->] (8,26) -- (0,32); 
\draw[colarr4,->] (8,28) -- (0,35); 
\draw[colarr4,->] (8,34) -- (0,40); 
\draw[colarr4,->] (8,34.5) -- (0,40.5); 

\draw[colarr4,->] (20, 2) -- (12, 8); 
\draw[colarr4,->] (20,10) -- (12,16.0);
\draw[colarr4,->] (20,12) -- (12,18.5); 
\draw[colarr4,->] (20,18) -- (12,24); 
\draw[colarr4,->] (20,20) -- (12,26.5);  
\draw[colarr4,->] (20,26) -- (12,32); 
\draw[colarr4,->] (20,28) -- (12,35); 
\draw[colarr4,->] (20,34) -- (12,40); 
\draw[colarr4,->] (20,34.5) -- (12,40.5); 

\draw[colarr4,->] (32, 2) -- (24, 8); 
\draw[colarr4,->] (32,10) -- (24,16.0);
\draw[colarr4,->] (32,12) -- (24,18.5); 
\draw[colarr4,->] (32,18) -- (24,24); 
\draw[colarr4,->] (32,20) -- (24,26.5);  
\draw[colarr4,->] (32,26) -- (24,32); 
\draw[colarr4,->] (32,28) -- (24,35); 
\draw[colarr4,->] (32,34) -- (24,40); 
\draw[colarr4,->] (32,34.5) -- (24,40.5); 

\draw[colarr4,->] (44, 2) -- (36, 8); 
\draw[colarr4,->] (44,10) -- (36,16.0);
\draw[colarr4,->] (44,12) -- (36,18.5); 
\draw[colarr4,->] (44,18) -- (36,24); 
\draw[colarr4,->] (44,20) -- (36,26.5);  
\draw[colarr4,->] (44,26) -- (36,32); 
\draw[colarr4,->] (44,28) -- (36,35); 
\draw[colarr4,->] (44,34) -- (36,40); 
\draw[colarr4,->] (44,34.5) -- (36,40.5);

\draw[colarr6,->] (12, 0) -- (0, 10.5); 
\draw[colarr6,->] (12, 32.5) -- (0,42.5);
\draw[colarr6,->] (12, 35.5) -- (0, 45.5);

\draw[colarr6,->] (24, 0) --(12, 10.5); 
\draw[colarr6,->] (24, 32.5) -- (12,42.5);
\draw[colarr6,->] (24, 35.5) -- (12, 45.5);

\draw[colarr6,->] (48, 0) --(36, 10.5); 
\draw[colarr6,->] (48, 32.5) -- (36,42.5);
\draw[colarr6,->] (48, 35.5) -- (36, 45.5);

\draw[colarr12,->] (24, 10.5) -- (0, 32.5); 

\draw[colarr18,->] (36, 0) -- (0, 35.5); 

\end{tikzpicture}

\caption{The spectral sequence for $\pi_{*,*} MHH(M\mathbb{Z}/3)[\tau^{-1}]$ in low degrees. Each point represents a generator of a module isomorphic to $\mathbb{F}_3[\tau^{\pm 1}]$.}

\end{center}

\end{figure}

As before, we begin with an inspection of what happens in low degrees. Observe that properties analogous to \ref{rmk:key.properties.spec.seq} hold in this setting as well.

The first observations coincide with those for $p=2$: we can conclude that we have the permanent cycle $1$ in degree $(0,0)$ and that $\pi_{1,*} MHH(M\mathbb{Z}/p)[\tau^{-1}]\cong 0$.

Next, as we have $\pi_{1,*}\mathcal{A}(p)[\tau^{-1}] \cong \mathbb{F}_p[\tau^{\pm 1}]\{\tau_0\}$, we need an element $\mu_0 \in \pi_{2,0} MHH(M\mathbb{Z}/p)[\tau^{-1}]$ with $d_2(\mu_0)=\tau_0$. By the Leibniz rule, we get an isomorphism of modules:
\begin{equation*}
\mathbb{F}_p[\tau^{\pm 1}]\{ \mu_0 \} \xrightarrow{d^2} \mathbb{F}_p[\tau^{\pm 1}]\{ \tau_0 \} \cong E^2_{0,1,*}.
\end{equation*}
As we must not introduce permanent cycles, we conclude:
\begin{equation}
\pi_{2,*} MHH(M\mathbb{Z}/p)[\tau^{-1}]\cong \mathbb{F}_p[\tau^{\pm 1}]\{ \mu_0 \}.
\end{equation}

We then observe that the $d^2$ differential kills all the elements in $\tau_0\pi_{*,*} \mathcal{A}(p) \subset E^2_{(0,*)}$, as we have:
\begin{equation}
d^2(\mu_0\alpha)=\tau_0\alpha
\end{equation}
for every $\alpha \in \pi_{*,*} \mathcal{A}(p) $.
On the other hand, as we have the relation $\tau_0^2=0$ in the dual Steenrod algebra, some elements in $E^2_{2,*,*}$ have a trivial $d^2$ differential, namely, those in $\mu_0 \tau_0 \pi_{*,*}\mathcal{A}(p)$. 

Next, we notice that $E^2_{3,0,*}$ must be a trivial module to avoid permanent cycles, hence:
\[
    \pi_{3,*} MHH(M\mathbb{Z}/p)[\tau^{-1}] \cong 0.
\]
Consequently, $E^k_{3,*,*}\cong 0$ for all $k \geq 2$.

In degree $\pi_{4,0} MHH(M\mathbb{Z}/p)[\tau^{-1}]$ we have the element $\mu_0^2$, which has $d^2$ differential $2\mu_0\tau_0$. Hence, by the Leibniz rule (we are working over $\mathbb{F}_p$ for $p$ odd), we have an isomorphism:
\[
    \mathbb{F}_p[\tau^{\pm 1}]\{ \mu_0^2 \} \xrightarrow{d^2} \mathbb{F}_p[\tau^{\pm 1}]\{ \mu_0\tau_0 \} \cong E^2_{2,1,*}.
\]

As the first column and the rows $E^2_{*,k,*}$ for $2 \leq k \leq 2p-3$ are empty, to avoid any permanent cycle we must have
\[
    \pi_{4,*} MHH(M\mathbb{Z}/p)[\tau^{-1}] \cong  \mathbb{F}_p[\tau^{\pm 1}]\{ \mu_0^2 \}
\]
The differential behavior of $\mu_0^2$ propagates via the Leibniz rule to the whole column $E^2_{4,*,*}$; the image of such differentials is the $E^2_{0,*,*}$-submodule of $E^2_{2,*,*}$ generated by $\tau_0$. Hence, the whole column $E^3_{2,*,*}\cong 0$ is trivial.

This process repeats identically up to first degree $2(p-1)$:
\begin{equation*}
\begin{array}{l}
\pi_{2k,*} MHH(M\mathbb{Z}/p)[\tau^{-1}]\cong \mathbb{F}_p[\tau^{\pm 1}]\{ \mu_0^k \}\\
\pi_{2k-1,*} MHH(M\mathbb{Z}/p)[\tau^{-1}]\cong 0.
\end{array} 
\hspace{1cm} k=1,\ldots,p-1.
\end{equation*}
For $k=1,\ldots,p-1$, $\mu_0^k$ supports a $d^2$-differential: $d^2(\mu_0^k)= k \mu_0^{k-1} \tau_0$; hence $E^3_{s,*,*}\cong 0$ for $1 \leq s \leq 2p-3$.

At this point, we have to take into account the element $\xi_1\in \pi_{2(p-1),p-1}\mathcal{A}(p)$, as some differentials might hit it. Namely, we have two possibilities (given in particular that the columns $E^3_{s,*,*}$ for $1 \leq s \leq 2p-3$ are empty):
\begin{itemize}
    \item The element $\mu_0^{p-1}\tau_0$ survives up to the $E^{2(p-1)}$ page and it originates a $d^{2(p-1)}$ differential: $d^{2(p-1)}(\mu_0^{p-1}\tau_0)=\tau^{p-1}\xi_1.$
    \item There are an element $\lambda_1 \in \pi_{2p-1,p-1} MHH(M\mathbb{Z}/p)[\tau^{-1}]$ surviving up to the $E^{2p-1}$ page and a $d^{2p-1}$ differential: $d^{2p-1}(\lambda_1)=\xi_1.$
\end{itemize}
To solve this issue, we make use of \ref{cor:power.operations}: we know that $\sigma_* \xi_1 = 0$, so it cannot be in the image of a non-trivial transgressive differential. Hence:
\[
    d^{2(p-1)}(\mu_0^{p-1}\tau_0)=\tau^{p-1}\xi_1.
\]

Moving on to degree $(2p-1,0)$, we see that any element here cannot have non-trivial differentials, so it would produce permanent cycles; so we must have:
\begin{equation*}
    \pi_{2p-1,*} MHH(M\mathbb{Z}/p)[\tau^{-1}]\cong 0.
\end{equation*}

In the next horizontal degree, we encounter the element $\mu_0^p$; it has a trivial $d^2$ differential by characteristic reasons. On the other hand, by what we have determined so far, we need an element $\mu_1 \in \pi_{2p,p-1} MHH(M\mathbb{Z}/p)[\tau^{-1}]$ surviving up to the $E^{2p}$ page and a differential $d^2(\mu_1)=\tau_1$. Thanks to corollary \ref{cor:power.operations}, we can make the identification:
\[
    \mu_1=\sigma_* \tau_1= \tau^{-p+1}(\sigma_* \tau_0)^p=\tau^{-p+1}\mu_0^p.
\]
As usual, this extends linearly to an isomorphism:
\[
    \mathbb{F}_p[\tau^{\pm 1}]\{ \mu_0^p \} \xrightarrow{d^{2p}} \mathbb{F}_p[\tau^{\pm 1}]\{ \tau_1 \} \cong E^{2p}_{0,2p-1,*}
\]
that allows us to conclude:
\[
    \pi_{2p,*} MHH(M\mathbb{Z}/p)[\tau^{-1}] \cong  \mathbb{F}_p[\tau^{\pm 1}]\{ \mu_0^p \}.
\]

Following this pattern, by considering higher powers of $\mu_0$ and differentiating according to the Leibniz rule, one uncovers the structure of $\pi_{*,*} MHH(M\mathbb{Z}/p)[\tau^{-1}]$ in higher degrees; in particular, after considering enough degrees, one can confirm that all the classes in $E^2_{2(p-1),*,*}$ support either a non-trivial $d^2$ or a non-trivial $d^{2(p-1)}$-differential, and are not in the image of any differential. After this page, $E^{2p-1}_{2(p-1),*,*} \cong 0$. Similarly, one confirms that the classes in $E^2_{2p,*,*}$ that are not hit by a $d^2$-differential, and hence survive to the $E^3$-page, are either hit by a $d^{2p}$-differential from the $4p$-th column (we have $E^2_{4p,*,*} \cong E^2_{0,*,*} \{ \mu_0^{2p}\}$) or support a non-trivial $d^{2p}$-differential. In any case, $E^{2p+1}_{2p,*,*} \cong 0$.

The next step that requires further investigation is when one needs to deal with $\xi_2\in \pi_{2(p^2-1),p^2-1}\mathcal{A}(p)$. At the same time, we notice that the element $\mu_0^{p(p-1)}\tau_1 \in E^{2(p-1)}_{2p(p-1),2p-1, p-1}$ has
\[
    d^{2(p-1)}(\mu_0^{p(p-1)}\tau_1)=(p-1)\mu_0^{p(p-2)}\tau_1^2=0
\]
As for $\xi_1$, we need to add an ``extra rule'' to our spectral sequence, and, as before, we just have two options:
\begin{itemize}
    \item The element $\mu_0^{p(p-1)}\tau_1$ survives up to the $E^{2p(p-1)}$ page and originates a $d^{2p(p-1)}$ differential: $d^{2p(p-1)}(\mu_0^{p(p-1)}\tau_1)=\tau^{p^2-1}\xi_2.$
    \item There are an element $\lambda_2 \in \pi_{2p^2-1,p^2-1} MHH(M\mathbb{Z}/p)[\tau^{-1}]$ surviving up to the $E^{2p-1}$ page and a $d^{2p-1}$ differential: $d^{2p^2-1}(\lambda_1)=\xi_1.$
\end{itemize}
We exclude the second option again thanks to \ref{cor:power.operations}, as $\xi_2$ is not transgressive, so there must be a differential:
\[
    d^{2p(p-1)}(\mu_0^{p(p-1)}\tau_1)=\tau^{p^2-1}\xi_2
\]
which by linearity produces an isomorphism of modules:
\[
    \mathbb{F}_p[\tau^{\pm 1}]\{ \mu_0^{p(p-1)}\tau_1 \} \xrightarrow{d^{2p(p-1)}} \mathbb{F}_p[\tau^{\pm 1}]\{ \xi_2 \} \cong E^{2p}_{0,2(p^2-1),*}.
\]

This hints at how the spectral sequence works in general:
\begin{prop}\label{prop:MHH(p).[tau-1]}
    $\pi_{*,*} MHH(M\mathbb{Z}/p)[\tau^{-1}]$ is an $\mathbb{F}_p[\tau^{\pm 1}]$ algebra in one single polynomial generator $\mu_0$ of degree $(2,0)$. \\
    If $n$ is any integer coprime with $p$, then $\mu_0^{np^i}$ survives in the spectral sequence up to the $2p^i$-page, where it supports a differential $d^{2p^i}(\mu_0^{np^i})=n\mu_0^{(n-1)p^i} \tau_i$. Elements of the form $\mu_0^{n(p-1)p^i}\tau_i\in E^2$ survive up to the $E^{2(p-1)p^i}$ page, where they support a differential $d^{2(p-1)p^i}(\mu_0^{n(p-1)p^i}\tau_i)=n\mu_0^{(n-1)(p-1)p^i}\xi_{i+1}$. 
\end{prop}

This allows in fact to explicit the behaviour of any element in the $E^2$ page:
\begin{cor}\label{cor:overwhelming.p.odd}
The elements in $E^2_{s,t,*}=\pi_{s,*} MHH(M\mathbb{Z}/p)[\tau^{-1}] \otimes_{\mathbb{F}_p[\tau^{\pm 1}]} \pi_{t,*} \mathcal{A}(p)$ are $\mathbb{F}_p[\tau^{\pm 1}]$ linear combinations of elements of the form: $\mu_0^{np^i}\tau_K\xi_J^{e_J}$, where:
\begin{itemize}
    \item $J=(j_1,j_2,\ldots,j_M)$ and $ K=(k_1,k_2,\ldots,k_N)$ are multi-indices of distinct positive integers; 
    \item $e_J=(e_{j_1},e_{j_2},\ldots,e_{j_M})$ is a  multi-index of positive integers;
    \item $\tau_K=\tau_{k_1}\tau_{k_2}\cdots \tau_{k_N}$ and $\xi_J^{e_J}=\xi_{j_1}^{e_{j_1}}\xi_{j_2}^{e_{j_2}}\cdots \xi_{j_M}^{e_{j_M}}$;
    \item $GCD(n,p)=1$.
\end{itemize}

The degrees are: $|\mu_1|=(2,0,0)$, $|\tau_k|=(0,2p^k-1,p^k-1)$ and $|\xi_j|=(0,2(p^j-1),p^k-1)$.

The element $\mu_0^{np^i}\tau_K\xi_J^{e_J}$ has the following behaviour in the spectral sequence: let $h=min(i,K,J)$.
\begin{itemize}
    \item If $i=h$, $n+1 \equiv 0\, (mod.\, p)$ and a $\tau_i$ appears in the decomposition, the element survives to the $E^{2p^i(p-1)}$ page and differentiates:
    \begin{equation}
        d^{2p^i(p-1)}(\mu_0^{np^i}\tau_K\xi_J^{e_J})=\frac{\mu_0^{(n-p+1)p^i}\tau_K\xi_J^{e_J}\xi_{i+1}}{\tau_i}
    \end{equation}
    \item Otherwise, the following applies:
    \begin{itemize}
        \item $\mu_0^{np^h}$ is involved in a $d^{2p^h}$  exiting differential: 
        \begin{equation}
            d^{2p^h}(\mu_0^{np^h})=n\mu_0^{(n-1)p^h}\tau_h;
        \end{equation}
        \item $\tau_h$ is involved in a $d^{2p^{h}}$ entering differential: 
        \begin{equation}
            d^{2p^h}(\mu_0^{p^h})=\tau_h;
        \end{equation}
        \item $\xi_h$ is involved in a $d^{2(p-1)p^{h-1}}$  entering differential:
        \begin{equation}
            d^{2(p-1)p^{h-1}}(\mu_0^{(p-1)p^{h-1}\tau_{h-1}})=\xi_{h}.
        \end{equation}
    \end{itemize}
    The factor determining the index $h$ determines the behaviour of the product. In case of multiple elements with the same index $h$, $\xi_h$ prevails on $\tau_h$ that prevails on $\mu_0^{np^h}$.
\end{itemize}
\end{cor}

\begin{proof}[Proof of \ref{prop:MHH(p).[tau-1]}]
The proof is similar to the case $p=2$, but it presents the additional complication of a richer structure in the dual motivic Steenrod algebra, which is paralleled by a greater variety of differentials.

We proceed by induction on a natural number $k\geq 1$, showing that there is an isomorphism of graded rings:
\begin{equation}\label{eq:isomorphism.pi.MHH.p}
    \pi_{*,*}MHH(M\mathbb{Z}/p)[\tau^{-1}] \cong \mathbb{F}_p[\tau^{\pm 1}, \mu_0]
\end{equation}
with $|\tau|=(0,-1)$, $|\mu_0|=(2,0)$, for the first degree $s \leq 2p^k-1$ and all weights. At the same time, we show that the behaviour of these elements in the spectral sequence has to correspond to what we claimed. As we will see, this, together with the Leibniz rule, will force uniquely the behaviour in the spectral sequence of all the columns $E^{2}_{s,*,*}$ for $1 \leq s \leq 2p^k-1$, as described in corollary \ref{cor:overwhelming.p.odd}.

Some base steps of this induction argument were already carried out above. 

The induction hypothesis at $k$ implies that all modules in the covered range will be trivial after the $E^{2(p-1)p^{k-1}}$-page. This is because all elements in the region either are hit by non-trivial differentials starting in this region or support non-trivial differentials landing in this region or the zeroth column. This implies that there can't be any other non-trivial differential, in particular landing in this region from outside, as this would interfere with the necessary differentials, producing unwanted permanent cycles (recall: the convergence term of the spectral sequence is concentrated in degree $(0,0,*)$).

As our main job will consist of proving that the Leibniz rule is respected, we begin by making a small observation on this.
\begin{rmk}[On the violation of the Leibniz rule]\label{rmk:on.the.violation.of.Leibniz}
    Consider a product of homogeneous classes $\alpha\beta \in E^i$ with $d^i(\beta)=0$ and $d^i(\alpha) \neq 0$; suppose moreover that this induces an isomorphism of $\mathbb{F}_p[\tau^{\pm 1}]$-modules:
    \[
        \mathbb{F}_p[\tau^{\pm 1}]\{ \alpha \} \xrightarrow{d^i}\mathbb{F}_p[\tau^{\pm 1}]\{ d^i(\alpha) \}.
    \]
    Then the Leibniz rule prescribes a differential $d^i(\alpha\beta)=d^i(\alpha)\beta$. This cannot be avoided; however, this second differential could fail to be an isomorphism of $\mathbb{F}_p[\tau^{\pm 1}]$-modules (thus potentially producing non-zero classes in the homology) if either the source or target present some torsion at the $E^i$-page. As any module is born free at the $E^2$ page, this requires involving the source and/or the target with shorter differentials. As if $\alpha \beta$ presents some $\mathbb{F}_p[\tau^{\pm 1}]$-torsion at the $E^i$ page (even $\alpha \beta =0$), by linearity of the differential also $d^i(\alpha)\beta$ has to present the same torsion, so we can just focus on the target: there must be at least one third element $\gamma$ in the $E^2$-page, living between $d^i(\alpha)\beta$ and $\alpha\beta$, that generates a module supporting at a certain page a non-trivial differential hitting a $\mathbb{F}_p[\tau^{\pm 1}]$-module contributing to former\footnote{Recall that higher pages in the spectral sequence are quotients of submodules of lower pages.}. 
\end{rmk}
Thus, specifically, if we identify a differential departing from the horizontal line at a particular page, it suffices to understand what happens in the spectral sequence at lower degrees to ascertain whether the behaviour of the elements in the corresponding column adheres to the Leibniz rule, or if there are other differentials at play, originating from higher degrees.

Thanks to the following claim (confront with claim \ref{clm:in.the.proof.of.prop:pi.MHH(MZ/2)[tau-1]}), we see that the behaviour of powers of $\mu_0$ is in fact determined ``locally'':

\begin{claim} \label{clm:in.the.proof.of.prop:pi.MHH(MZ/p)[tau-1]}
    Consider a number of the form $mp^{j+1}+ap^j$ for some $m \geq 0$ and $0 \leq a \leq p-1$ that lies in the range $p^k \leq mp^{j+1}+ap^j \leq p^{k+1}-1$ (in particular $j\leq k$). Under the induction hypothesis for $k$, suppose that $\mu_0^{mp^{j+1}+ap^j}$  generates a free $\mathbb{F}_p[\tau^{\pm 1}]$-module and has trivial differentials up to (at least) $d^{2p^{j}}$; suppose also that we there are no elements on the zeroth line with degree  $2(mp^{j+1}+ap^j)<d<2(mp^{j+1}+ap^j+p^j-1)$ supporting non-trivial differentials $d^l$ with $l>d-2(mp^{j+1}+ap^j)$ (in other words, non-trivial $d^l$-differentials crossing the $2(mp^{j+1}+ap^j)$-th column).
    
    Then each the product $\mu_0^{mp^{j+1}+ap^j+h}$, for $0 \leq h \leq p^j-1$, generates a free $\mathbb{F}_p[\tau^{\pm 1}]$-module; these elements and all the ones in the columns above them differentiate according to or are the image of differentials determined by the Leibniz rule (hence similarly to what happens for the elements in $E^2_{0,*,*}\{\mu_0^h\}$). In particular, this produces isomorphisms of $E^{i}_{0,*,*}$-modules: 
    \[
        E^{i}_{0,*,*} \xrightarrow{\cdot \mu_0^{mp^{j+1}+ap^j}} E^{i}_{2(mp^{j+1}+ap^j),*,*}
    \]
    for $2 \leq i \leq 2p^j$; moreover all columns $E^{2p^{j-1}+1}_{2(mp^{j+1}+ap^j)+l,*,*}\cong 0$ for $1 \leq l \leq 2p^j-1$.
\end{claim}

The proof is carried out below.

We proceed now with the induction step, assuming the thesis up to some integer $k$ and this claim. We start by observing that $d^{2p^{k-1}}(\mu_0^{p^k})=0$ because of the Leibniz rule, so $\mu_0^{p^k}$ survives to higher pages. The induction hypothesis already fixes the structure and the differentials for everything in positive degree left to $\mu_0^{p^k}$, so $\mu_0^{p^k}$ cannot support non-trivial differentials landing in this region. Given the convergence of the spectral sequence, the only possibilities left are either a non-trivial $d^{2p^k}$-differential to the module generated by $\tau_h$ or a torsion relation $\mu_0^{p^k}=0$. In fact, the former is true, as corollary \ref{cor:power.operations} gives:
\[
    \sigma_* \tau_k =\tau^{-p+1}(\sigma_* \tau_{k-1})^p=\tau^{-p^k+1}\mu_0^{p^k}.
\]
(Also because otherwise $\tau_k$ would be a permanent cycle). By linearity and the convergence of the spectral sequence, this must extend to an isomorphism of modules:
\[
    \mathbb{F}_p[\tau^{\pm 1}]\{\mu_0^{p^k}\} \xrightarrow{d^{2p^k}} \mathbb{F}_p[\tau^{\pm 1}]\{\tau_k\}
\]
As nothing hits $\tau_k$ in the previous pages, we conclude that $\mu_0^{p^k}$ generates a free module. As $E^{2p^k}_{0,2p^k-1,*}$ is generated over $\mathbb{F}_p[\tau^{\pm 1}]$ by the sole $\tau_k$, we conclude an isomorphism:
\[
    E^{2p^k}_{2p^k,0,*} \cong \mathbb{F}_p[\tau^{\pm 1}]\{\mu_0^{p^k}\}
\]
is a free module $\mathbb{F}_p[\tau^{\pm 1}]$ on one generator. As we already remarked, nothing but a transgressive differential can start from position $(2p^k,0,*)$, so this isomorphism has to come from an isomorphism in the $E^2$-page:
\[
    \pi_{*,*}MHH(M\mathbb{Z}/p)[\tau^{-1}]\cong \mathbb{F}_p[\tau^{\pm 1}]\{\mu_0^{p^k}\}.
\]
Observe now that the next non-zero class at this page in the zeroth column is $\xi_{k+1}$, which has degree $2p^{k+1}-2$; as again any new class, right of $\mu_0^{p^k}$, with a nontrivial differential left to hit would have to hit the zeroth column, this class should have degree at least $2p^{k+1}-1$: this excludes the presence of such elements in degree smaller than or equal to $2p^{k+1}-2$. We can then apply claim \ref{clm:in.the.proof.of.prop:pi.MHH(MZ/p)[tau-1]} to $\mu_0^{p^k}$ (with $j=k$) and conclude that the structure of the powers of $\mu_0$ up to $\mu_0^{2p^k-1}$ and their behaviour in the spectral sequence has to comply with our thesis.
Observe also that the conclusion of the claim does not give room for alien classes in degrees between $2p^k$ and $4p^k-1$, because of the vanishing conditions. In degree $4p^k$ we encounter $\mu_0^{2p^k}$; the absence of shorter differentials striking $\mu_0^{p^k}$ (again in the conclusions of the claim) gives that the differential coming from the Leibniz rule produces an isomorphism of $\mathbb{F}_p[\tau^{\pm 1}]$ modules:
\[
    \mathbb{F}_p[\tau^{\pm 1}]\{\mu_0^{2p^k}\} \xrightarrow{d^{2p^k}} \mathbb{F}_p[\tau^{\pm 1}]\{\mu_0^{p^k}\tau_k\}
\]
We can then apply the claim to $\mu_0^{2p^k}$ (again with $j=k$) and draw analogous conclusions.
In particular, we get:
\[
    d^{2p^k}(E^{2p^k}_{4p^k,*,*})=\tau_kE^{2p^k}_{2p^k,*,*}=ker(d^{2p^k}_{|E^{2p^k}_{2p^k,*,*}})
\]
In other words, the column $E^{2p^k}_{2p^k,*,*}$ disappears from the next page.
We proceed exploring all powers of $\mu_0^{p^k}$ and drawing similar conclusions, up to the power $\mu_0^{(p-1)p^k}$. When looking at the next power, $\mu_0^{p\cdot p^k}=\mu_0^{p^{k+1}}$, we could potentially draw the same conclusions, but its image under the $d^{2p^k}$ differential is, because of the Leibniz rule, trivial. So the kernel of the $d^{2p^k}$ differential restricted to the $E^{2p^k}_{(p-1)p^k,*,*}$ column, corresponding to the $E^{2p^k}_{0,*,*}$-submodule generated by $\mu_0^{(p-1)p^k}\tau_k$, is not in its image. This submodule (in particular its generator) has however to die, because of the convergence condition. The rigidity given by the claim \ref{clm:in.the.proof.of.prop:pi.MHH(MZ/p)[tau-1]} does not allow for further elements of degree smaller than or equal to $2p^{k+1}-2$. Now recall that we also have to deal with the element $\xi_{k+1}$ in the zeroth column. There are in fact only two possibilities to kill this $\xi_{k+1}$:
\begin{itemize}
    \item We introduce a new element $\alpha \in \pi_{2p^{k+1}-1,*}MHH(M\mathbb{Z}/p)[\tau^{\pm 1}]$ which survives to the $E^{2p^{k+1}-1}$ page and supports a transgressive non-trivial differential $\alpha \mapsto \xi_{h+1}$. In this case $\mu_0^{(p-1)p^k}\tau_k$ could either be in the image of a differential coming from $\alpha\tau_0$ or from some other element $\beta \in \pi_{2p^{k+1},*}MHH(M\mathbb{Z}/p)[\tau^{\pm 1}]$.
    \item The other option is that $\mu_0^{(p-1)p^{h-1}}\tau_{h-1}$ survives up to the $E^{2(p-1)p^{h-1}}$-page and $d^{2(p-1)p^{h-1}}(\mu_0^{(p-1)p^{h-1}}\tau_{h-1})=\xi_h$. In this case, no other element has to be introduced.
\end{itemize} 
The first one is excluded since we know from \ref{cor:power.operations} that $\sigma_*(\xi_h)=0$ is not transgressive. But then we can conclude also that $E^2_{2p^{k+1}-1,*,*} \cong 0$.

This concludes the proof of the induction step.
\end{proof}

\begin{proof}[Proof of \ref{clm:in.the.proof.of.prop:pi.MHH(MZ/p)[tau-1]}]
    By induction on $0 \leq j \leq k$. 

    The case $j=0$ corresponds to powers of $\mu_0$ with exponent coprime with $p$: here we have nothing to prove, as they support a non-trivial $d^2$ differential by the Leibniz rule, and the spectral sequence starts at the $E^2$ page.

    Suppose the claim holds for all possible indices up to a certain $j-1$. Given $\mu_0^{mp^{j+1}+ap^j}$ as in the claim, we make a finite double induction, first on an index $0 \leq n \leq j-1$; then for each $n$, we consider the elements  $\mu_0^{mp^{j+1}+ap^j+bp^n}$ for $b=1,\ldots,p-1$. They all satisfy the hypothesis of this argument. In fact, because of the Leibniz rule, we know that $\mu_0^{mp^{j+1}+ap^j+bp^n}$ has trivial differentials up to $d^{2p^n}$; moreover, the thesis for $\mu_0^{mp^{j+1}+ap^j+(b-1)p^n}$ (in other words, the hypothesis from which we begin a certain induction step) implies that the $d^{2p^n}$-differential arising from the Leibniz rule:
    \[
        \mu_0^{mp^{j+1}+ap^j+bp^n} \xrightarrow{d^{2p^n}} \mu_0^{mp^{j+1}+ap^j+(b-1)p^n} \tau_n
    \]
    induces an isomorphism of $\mathbb{F}_p[\tau^{\pm 1}]$-modules:
    \[
        \mathbb{F}_p[\tau^{\pm 1}]\{\mu_0^{mp^{j+1}+ap^j+bp^n}\} \xrightarrow{d^{2p^{n}}} \mathbb{F}_p[\tau^{\pm 1}]\{\mu_0^{mp^{j+1}+ap^j+(b-1)p^n} \tau_n\}
    \]
    In particular, the source cannot present $\mathbb{F}_p[\tau^{\pm 1}]$-torsion, in other words, is free. The absence of ``extra'' elements follows from a combination of the hypothesis of the claim (there are no ``extra'' elements supporting very long differentials) and a vanishing result on the columns involved in the argument (there are no ``extra'' elements supporting short differentials), pretty much in the same way as discussed in the main proof. For $b \leq p-1$ the thesis follows by the statement of the claim for $n<j-1$, which we know to hold by induction. Notice in particular that the $d^{2p^n}$ differential is induced by the Leibniz rule:
    \begin{multline*}
        E^{2p^n}_{0,*,*}\{\mu_0^{mp^{j+1}+ap^j+bp^n}\} \cong E^{2p^n}_{2(mp^{j+1}+ap^j+bp^n),*,*} \\
        \xrightarrow{d^{2p^n}} E^{2p^n}_{2(mp^{j+1}+ap^j+(b-1)p^n),*,*} \cong  E^{2p^n}_{0,*,*}\{\mu_0^{mp^{j+1}+ap^j+(b-1)p^n}\}
    \end{multline*}
    surjects onto the ideal generated by $\tau_n$. This implies the vanishing $E^{2p^n+1}_{2(mp^{j+1}+ap^j+(b-1)p^n),*,*} \cong 0$ for $1 \leq b-1 \leq p-2$.
    We can then move on to the next value of $b$. 

    When we get the statement for $b=p-1$, we reach the following situation in low rows: at the $E^{2^{p^n}}$ page, we find, among others, the non-zero elements $\mu_0^{mp^{j+1}+ap^j+(p-1)p^n}\tau_n$, $\mu_0^{mp^{j+1}+ap^j} \xi_{n+1}$ and  $\mu_0^{mp^{j+1}+ap^j} \tau_{n+1}$. The Leibniz rule provides a differential 
    \[
        \mu_0^{mp^{j+1}+ap^j+(p-1)p^n}\tau_n \mapsto \mu_0^{mp^{j+1}+ap^j} \xi_{n+1}.
    \]
    To prove that it is non-trivial, we have to exclude any shorter differential hitting $\mu_0^{mp^{j+1}+ap^j} \xi_{n+1}$ (see remark \ref{rmk:on.the.violation.of.Leibniz}) supported by an element $\alpha$ strictly between the columns on $\mu_0^{mp^{j+1}+ap^j}$ (which is the target) and $\mu_0^{mp^{j+1}+ap^j+(p-1)p^n}$ (where the Leibniz rule intervenes); this however cannot happen. In fact, all the classes we have already determined make the considered region empty after the $2p^n+1$-th page; moreover, no differential starting in this region crosses the $2(mp^{j+1}+ap^j)$-th column (that of $\mu_0^{mp^{j+1}+ap^j}$), so that all elements from powers of $\mu_0$ either hit this column (but not $\mu_0^{mp^{j+1}+ap^j} \xi_{n+1}$) or are null before they could support a differential hitting this column. So they cannot be the $\alpha$ we are looking for. Also, introducing extra elements in $\pi_{*,*} MHH(M\mathbb{Z}/p)[\tau^{-1}]$ would then produce torsion on the horizontal axis (recall that long differentials are forbidden by assumption). So we can confirm that this $d^{2(p-1)p^n}$-differential provided by the Leibniz rule for $\mu_0^{mp^{j+1}+ap^j+(p-1)p^n}\tau_n$ is non-trivial. This differential in fact extends non-trivially to the whole $E^{2(p-1)p^{n}}_{2(mp^{j+1}+ap^j+(p-1)p^n),*,*} \cong E^{2p^n+1}_{2(mp^{j+1}+ap^j+(p-1)p^n),*,*}$, see \ref{rmk:on.the.violation.of.Leibniz}, as it consists of multiples of $\mu_0^{mp^{j+1}+ap^j+(p-1)p^n}\tau_n$. Hence, this whole column is trivial on the next page: $E^{2(p-1)p^{n}+1}_{2(mp^{j+1}+ap^j+(p-1)p^n),*,*} \cong 0$.
    
    Observe in particular that the whole region between degrees $2(mp^{j+1}+ap^j)+1$ and $2(mp^{j+1}+ap^j+p^{n+1})-1$ vanishes after the $E^{2(p-1)p^{n}+1}$ page. Moreover, all the various differentials combined annihilated the $E^{2p^{n}}_{0,*,*}$-submodule of $E^{2p^{n}}_{2(mp^{j+1}+ap^j),*,*}$ generated by $\tau_n$ and $\xi_{n+1}$, exactly as it happens for the zeroth column, at the same pages.
    
    We can then increase $n$ by one: $\mu_0^{mp^{j+1}+ap^j+p^{n+1}}$ supports the non-trivial $d^{2p^{n+1}}$ differential coming from the Leibniz rule, as the elements between source and target do not pose any obstruction, and so on. 

    When one reaches the end of the argument for $n=j-1$ and $b=p-1$, one sees that there is a unique possible structure for the horizontal line, given by successive powers of $\mu_0$. The module $E^{i}_{2(mp^{j+1}+ap^j),*,*}$ is in particular isomorphic to:
    \begin{itemize}
        \item $E^{2p^{j-1}+1}_{2(mp^{j+1}+ap^j),*,*}$ for $2p^{j-1}+1\leq i \leq 2(p-1)p^{j-1}$ (after the quotient by $\tau_{j-1}$)
        \item $E^{2(p-1)p^{j-1}+1}_{2(mp^{j+1}+ap^j),*,*}$, for $2(p-1)p^{j-1}+1 \leq i \leq 2p^{j}$ (after the quotient by $\xi_{j}$) 
    \end{itemize}
    This recovers the isomorphisms with the zeroth column in the corresponding pages. Finally, the various induction passages, combined with what we observed at the various $2(mp^{j+1}+ap^j+bp^n)$ columns, provide the desired vanishing result.
\end{proof}

Observe that in this setting only specific pages (namely those indexed by $2p^k$ for some integer $k$) admit a decomposition as an algebra on the horizontal axis times one on the vertical axis.

\begin{rmk}
    Using the Greenlees spectral sequence \ref{prop:original.Greenlees}, one can calculate the homotopy groups of topological Hochschild homology of the Eilenberg-MacLane spectrum of the field $\mathbb{F}_p$, $p$ any prime, starting from the homotopy groups of the (classical) dual Steenrod algebra \cite[Section 5]{Milnor1967}:
    \[
        \pi_*(\mathcal{A}_{top}(p))\cong \begin{cases}
            \mathbb{F}_2[\tau_i]_{i \geq 0} & \text{ for }p=2\\
            \mathbb{F}_p[\tau_i, \xi_{i+1}]_{i \geq 0}/\langle \tau_i^2 \rangle & \text{ for }p \text{ odd.}
        \end{cases}
    \]
    The element $\tau_i$ is in degree $2p^i-1$ and the element $\xi_i$ is in degree $2p^i-2$.
    Given the great similarity between the structures of the homotopy rings of the classical and motivic dual Steenrod algebras, it is perhaps unsurprising that:
    \[
        \pi_*(THH(H\mathbb{F}_p)) \cong \mathbb{F}_p[\mu_0]
    \]
    with $\mu_0$ in degree 2. 
    One can set up a proof of this following the same steps that were used in this paper for the motivic context, obtaining identical spectral sequences (but for the base ring).
    This is in essence due to $\tau$ being invertible: whenever in this paper two classes where separated in weight, we were able connect them by an adequate power of $\tau$ (see for instance \ref{cor:power.operations}).
\end{rmk}

\clearpage

\clearpage

\printbibliography[heading=bibintoc]

\end{document}